\documentclass{amsart}
\usepackage{amssymb}
\usepackage{amsfonts}
\usepackage{hyperref}

\newtheorem{theorem}{Theorem}
\theoremstyle{plain}

\newtheorem{corollary}{Corollary}

\newtheorem{lemma}{Lemma}

\newtheorem{proposition}{Proposition}
\newtheorem{remark}{Remark}

\numberwithin{equation}{section}

\begin{document}
\title[Extension of a complete monotonicity theorem]{Extension of a complete
monotonicity theorem with applications.}
\author{Zhen-Hang-Yang}
\address{State Grid Zhejiang Electric Power Company Research Institute,
Hangzhou, Zhejiang, China, 3100142}
\email{yzhkm@163.com, yang\_zhenhang@zj.sgcc.com.cn}
\urladdr{https://orcid.org/0000-0002-2719-4728}
\date{June 22, 2025}
\thanks{This paper is in final form and no version of it will be submitted
for publication elsewhere.}

\begin{abstract}
Let $F_{p}\left( x\right) =L\left( t^{p}f\left( t\right) \right)
=\int_{0}^{\infty }t^{p}f\left( t\right) e^{-xt}dt$ converge on $\left(
0,\infty \right) $ for $p\in \mathbb{N}_{0}=\mathbb{N}\cup \left\{ 0\right\}
$, where $f\left( t\right) $ is positive on $\left( 0,\infty \right) $. In a
recent paper [Z.-H. Yang, A complete monotonicity theorem related to Fink's
inequality with applications, \emph{J. Math. Anal. Appl.} \textbf{551}
(2025), no. 1, Paper no. 129600], the author proved the sufficient
conditions for the function%
\begin{equation*}
x\mapsto \prod_{j=1}^{n}F_{p_{j}}\left( x\right) -\lambda
_{n}\prod_{j=1}^{n}F_{q_{j}}\left( x\right)
\end{equation*}%
to be completely monotonic on $\left( 0,\infty \right) $ by induction, where
$\boldsymbol{p}_{\left[ n\right] }=\left( p_{1},...,p_{n}\right) $ and $%
\boldsymbol{q}_{\left[ n\right] }=\left( q_{1},...,q_{n}\right) \in \mathbb{N%
}_{0}^{n}$ for $n\geq 2$ satisfy $\boldsymbol{p}_{\left[ n\right] }\prec
\boldsymbol{q}_{\left[ n\right] }$ However, the proof of the inductive step
is wrong. In this paper, we prove the above result also holds for $%
\boldsymbol{p}_{\left[ n\right] },\boldsymbol{q}_{\left[ n\right] }\in
\mathbb{I}^{k}$, where $\mathbb{I}\subseteq \mathbb{R}$ is an interval,
which extends the above result and correct the error in the proof of the
inductive step mentioned above. As applications of the extension of the
known result, four complete monotonicity propositions involving the Hurwitz
zeta function, alternating Hurwitz zeta function, the confluent
hypergeometric function of the second and Mills ratio are established, which
yield corresponding Tur\'{a}n type inequalities for these special functions.
\end{abstract}

\maketitle


\section{Introduction}

We begin with two concepts. The first one is the completely monotonic (%
\textbf{CM}) function, which states that a function $h$ is said to be CM on
an interval $I$ if $h$ has derivatives of all orders on $I$ and satisfies%
\begin{equation*}
\left( -1\right) ^{n}h^{\left( n\right) }\left( x\right) \geq 0\text{ \ for }%
n\in \mathbb{N}_{0}=\mathbb{N\cup }\left\{ 0\right\} \text{ and }x\in I.
\end{equation*}%
Obviously, if $f\left( x\right) $ and $g\left( x\right) $ are CM on $I$,
then $pf\left( x\right) +qg\left( x\right) $ for $p,q>0$ and $f\left(
x\right) g\left( x\right) $ are also CM on $I$ (see \cite[Theorem 1]%
{Miller-ITSF-12-2001}). It was proved in \cite[p. 161, Theorem 12b]%
{Widder-LT-1946} that the function $F(x)$ is CM on $\left( 0,\infty \right) $
if and only if
\begin{equation*}
F\left( x\right) =\int_{0}^{\infty }e^{-xt}d\mu \left( t\right) ,
\end{equation*}%
where $\mu \left( t\right) $\ is nondecreasing and the integral converges
for $0<x<\infty $, which, as usually, is called Bernstein theorem. Now if
let $d\mu \left( t\right) =f\left( t\right) dt$, then Bernstein theorem has
an equivalent statement: the function $F(x)$ is CM on $\left( 0,\infty
\right) $ if and only if
\begin{equation*}
F\left( x\right) =\int_{0}^{\infty }f\left( t\right) e^{-xt}dt,
\end{equation*}%
where $f\left( t\right) $\ is nonnegative on $\left( 0,\infty \right) $ and
the integral converges for $0<x<\infty $.

The second one is the majorization of vectors (see \cite[p. 8, Definition A.1%
]{Marshall-I-TMA-1979}). Let $\boldsymbol{x}_{\left[ n\right] }=\left(
x_{1},...,x_{n}\right) $\ and\ $\boldsymbol{y}_{\left[ n\right] }=\left(
y_{1},...,y_{n}\right) \in \mathbb{R}^{n}$. The $n$-tuple $\boldsymbol{x}_{%
\left[ n\right] }$ is said to be strictly majorized by another $n$-tuple $%
\boldsymbol{y}_{\left[ n\right] }$, denoted by $\boldsymbol{x}_{\left[ n%
\right] }\succ \boldsymbol{y}_{\left[ n\right] }$ or $\boldsymbol{y}_{\left[
n\right] }\prec \boldsymbol{x}_{\left[ n\right] }$, if $\boldsymbol{x}_{%
\left[ n\right] }\neq \boldsymbol{y}_{\left[ n\right] }$,
\begin{eqnarray*}
x_{1} &\geq &x_{2}\geq \cdot \cdot \cdot \geq x_{n}\text{, \ \ \ }y_{1}\geq
y_{2}\geq \cdot \cdot \cdot \geq y_{n}\text{,} \\
\sum_{j=1}^{n}x_{j} &\geq &\sum_{j=1}^{n}y_{j}\text{ for }j=1,...,n-1\text{
\ and \ }\sum_{j=1}^{n}x_{j}=\sum_{j=1}^{n}y_{j}.
\end{eqnarray*}

Very recently, inspired by Fink's inequality \cite{Fink-JMAA-90-1982}, Yang
\cite[Theorem 1.1]{Yang-JMAA-551-2025} considered the complete monotonicity
of the function%
\begin{equation}
\mathcal{D}_{\boldsymbol{p}_{\left[ n\right] },\boldsymbol{q}_{\left[ n%
\right] }}\left( x;\lambda _{n}\right) =\mathcal{D}_{\boldsymbol{p}_{\left[ n%
\right] },\boldsymbol{q}_{\left[ n\right] }}^{\left[ F\right] }\left(
x;\lambda _{n}\right) =\prod_{j=1}^{n}F_{p_{j}}\left( x\right) -\lambda
_{n}\prod_{j=1}^{n}F_{q_{j}}\left( x\right)  \label{Dpn,qn}
\end{equation}%
on $\left( 0,\infty \right) $, where%
\begin{equation}
F_{n}\left( x\right) =\left( -1\right) ^{n}F^{\left( n\right) }\left(
x\right) =\int_{0}^{\infty }t^{n}f\left( t\right) e^{-xt}dt\text{ for }n\in
\mathbb{N}_{0}\text{.}  \label{Fn}
\end{equation}%
For convenience, we records this result as follows (with $k=n$).

\begin{theorem}
\label{T0}Suppose that

(1) there exist constants $M>0$ and $c\geq 0$ such that $0<f\left( t\right)
\leq Me^{ct}$ for all $t>0$;

(2) the functions $F_{n}$ for $n\in \mathbb{N}_{0}$ defined by (\ref{Fn})
converge for all $x>0$;

(3) the $n$-tuples $\boldsymbol{p}_{\left[ n\right] }=\left(
p_{1},...,p_{n}\right) $ and $\boldsymbol{q}_{\left[ n\right] }=\left(
q_{1},...,q_{n}\right) \in \mathbb{N}_{0}^{n}$ for $n\geq 2$ satisfy $%
\boldsymbol{p}_{\left[ n\right] }\prec \boldsymbol{q}_{\left[ n\right] }$.

If there is real $\theta <1+\min \left\{ p_{n},q_{n}\right\} $ such that $%
f_{\theta }\left( t\right) =t^{\theta }f\left( t\right) $ is log-concave
(-convex) on $\left( 0,\infty \right) $, then the function $x\mapsto +\left(
-\right) \mathcal{D}_{\boldsymbol{p}_{\left[ n\right] },\boldsymbol{q}_{%
\left[ n\right] }}\left( x;\lambda _{n}^{\left[ \theta \right] }\right) $
define by (\ref{Dpn,qn}) is CM on $\left( 0,\infty \right) $, where%
\begin{equation}
\lambda _{n}^{\left[ \theta \right] }\equiv \lambda _{\boldsymbol{p}_{\left[
n\right] },\boldsymbol{q}_{\left[ n\right] }}\equiv \prod_{j=1}^{n}\frac{%
\Gamma \left( p_{j}-\theta +1\right) }{\Gamma \left( q_{j}-\theta +1\right) }%
.  \label{ln-thi}
\end{equation}
\end{theorem}

Now let us verify carefully the proof of \cite[Theorem 1.1]%
{Yang-JMAA-551-2025} and remark as follows.

\begin{remark}
(i) The proof of base case, i.e., \cite[Proof of Theorem 1.1 for $k=2$.]%
{Yang-JMAA-551-2025}, is correct.

(ii) \cite[Lemma 2.2]{Yang-JMAA-551-2025} also holds for $\boldsymbol{p}_{%
\left[ 2\right] },\boldsymbol{q}_{\left[ 2\right] }\in \mathbb{R}^{2}$, and
therefore, \cite[Proof of Theorem 1.1 for $k=2$.]{Yang-JMAA-551-2025} is
also valid for $\boldsymbol{p}_{\left[ 2\right] },\boldsymbol{q}_{\left[ 2%
\right] }\in \mathbb{I}^{2}$, where $\mathbb{I\subseteq R}$ is an interval,
provide $F_{p_{j}}\left( x\right) $ and $F_{q_{j}}\left( x\right) $ converge
on $\left( 0,\infty \right) $, where%
\begin{equation}
F_{p}\equiv F_{p}\left( x\right) =\int_{0}^{\infty }t^{p}f\left( t\right)
e^{-xt}dt\text{.}  \label{Fp}
\end{equation}

(iii) Unfortunately, the proof of inductive step, i.e., \cite[Proof of
Theorem 1.1 for all $k\in \mathbb{N}$ with $k\geq 2$.]{Yang-JMAA-551-2025}
is wrong. The error occurred on page 7 of \cite{Yang-JMAA-551-2025}, where
the last inequality in line 2 from the bottom can not be deduced by the
inequalities listed in line 4 from the bottom; in other words, the last
inequality in line 2 from the bottom does not hold. Thus, $\boldsymbol{p}_{%
\left[ k\right] }^{\ast }\nprec \boldsymbol{q}_{\left[ k\right] }$, and
hence, the inductive step failed.
\end{remark}

The purpose of this paper is to offer an extension of \cite[Theorem 1.1]%
{Yang-JMAA-551-2025} (extend $p_{j},q_{j}\in \mathbb{N}_{0}$ to $%
p_{j},q_{j}\in \mathbb{I}$, where $\mathbb{I\subseteq R}$ is an interval)
and to give a correct proof of it. The extension of \cite[Theorem 1.1]%
{Yang-JMAA-551-2025} can be stated as follows.

\begin{theorem}
\label{T1}Suppose that

(1) there is an interval $\mathbb{I\subseteq R}$ such that the function $%
F_{p}\left( x\right) $ defined by (\ref{Fp}) converges on $\left( 0,\infty
\right) $ for all $p\in \mathbb{I}$;

(2) there exist constants $M>0$ and $c\geq 0$ such that $0<t^{p}f\left(
t\right) \leq Me^{ct}$ for all $t>0$ and $p\in \mathbb{I}$;

(3) the $n$-tuples $\boldsymbol{p}_{\left[ n\right] }=\left(
p_{1},...,p_{n}\right) $ and $\boldsymbol{q}_{\left[ n\right] }=\left(
q_{1},...,q_{n}\right) \in \mathbb{I}^{n}$ for $n\geq 2$ satisfy $%
\boldsymbol{p}_{\left[ n\right] }\prec \boldsymbol{q}_{\left[ n\right] }$.

If there is a real $\theta <1+\min \left\{ p_{n},q_{n}\right\} $ such that $%
f_{\theta }\left( t\right) =t^{\theta }f\left( t\right) $ is log-concave
(-convex) on $\left( 0,\infty \right) $, then the function $x\mapsto +\left(
-\right) \mathcal{D}_{\boldsymbol{p}_{\left[ n\right] },\boldsymbol{q}_{%
\left[ n\right] }}\left( x;\lambda _{n}^{\left[ \theta \right] }\right) $
define by (\ref{Dpn,qn}) is CM on $\left( 0,\infty \right) $, where $\lambda
_{n}^{\left[ \theta \right] }$ is defined by (\ref{ln-thi}).
\end{theorem}

\begin{remark}
\label{R-T1}From the proofs of \cite[Lemma 2.2]{Yang-JMAA-551-2025} and \cite%
[Section 3.1]{Yang-JMAA-551-2025} we see that $x\mapsto -\mathcal{D}_{%
\boldsymbol{p}_{\left[ 2\right] },\boldsymbol{q}_{\left[ 2\right] }}\left(
x;1\right) $ is CM on $\left( 0,\infty \right) $. Using the same method as
the proof of Theorem \ref{T1}, we can prove that $x\mapsto -\mathcal{D}_{%
\boldsymbol{p}_{\left[ n\right] },\boldsymbol{q}_{\left[ n\right] }}\left(
x;1\right) $ is CM on $\left( 0,\infty \right) $. On the other hand, since $%
f_{\theta }\left( t\right) =t^{\theta }f\left( t\right) $ is log-convex on $%
\left( 0,\infty \right) $ as $\theta \rightarrow -\infty $ and by \cite[Eq.
(6.1.46)]{Abramowitz-HMFFGMT-1972},%
\begin{equation*}
\lambda _{n}^{\left[ \theta \right] }=\prod_{j=1}^{n}\frac{\Gamma \left(
p_{j}-\theta +1\right) }{\Gamma \left( q_{j}-\theta +1\right) }\sim
\prod_{j=1}^{n}\frac{\left( -\theta \right) ^{p_{j}+1}}{\left( -\theta
\right) ^{q_{j}+1}}=1\text{ as }\theta \rightarrow -\infty .
\end{equation*}%
Hence, Theorem \ref{T1} is also true if $\theta \rightarrow -\infty $, and
it implies Fink's inequality \cite[Theorem 1]{Fink-JMAA-90-1982}\ in this
limiting case.
\end{remark}

The rest of this paper is organized as follows. Two useful lemmas are given
in Section 2. The proof of Theorem \ref{T1} is arranged in Section 3. In
Section 4, we give four applications of Theorem \ref{T1} in the Hurwitz zeta
function, alternating Hurwitz zeta function, confluent hypergeometric
functions of the second kind, and Mills ratio, which extend \cite[%
Propositions 4.1, 5.1 and 6.1]{Yang-JMAA-551-2025}. In the last section, we
emphasize the significance of Lemma \ref{L-p,qn*,p,q2'}.

\section{Lemmas}

In order to complete the inductive step of proving Theorem \ref{T1} by
induction, we need the following lemma.

\begin{lemma}
\label{L-p,qn*,p,q2'}Let the $n+1$-tuples $\boldsymbol{p}_{\left[ n+1\right]
}=\left( p_{1},...,p_{n+1}\right) $ and $\boldsymbol{q}_{\left[ n+1\right]
}=\left( q_{1},...,q_{n+1}\right) \in \mathbb{R}^{n+1}$ for $n\geq 2$
satisfy $\boldsymbol{p}_{\left[ n+1\right] }\prec \boldsymbol{q}_{\left[ n+1%
\right] }$. Let%
\begin{equation*}
\boldsymbol{p}_{\left[ n\right] }^{\ast }=\left( p_{1}^{\ast
},...,p_{n}^{\ast }\right) \text{ \ and \ }\boldsymbol{q}_{[n]}^{\ast
}=(q_{1}^{\ast },...,q_{n}^{\ast }),
\end{equation*}%
where%
\begin{equation}
p_{j}^{\ast }=p_{j+1}\text{ \ and \ }q_{j}^{\ast }=\left\{
\begin{array}{ll}
q_{j} & \text{for }j=1,...,k-1\text{,} \\
q_{k}+q_{k+1}-p_{1} & \text{for }j=k\text{,} \\
q_{j+1} & \text{for }j=k+1,...,n\text{,}%
\end{array}%
\right.  \label{pj*,qj*}
\end{equation}%
and let%
\begin{equation}
\boldsymbol{p}_{\left[ 2\right] }^{\prime }=\left( p_{1}^{\prime
},p_{2}^{\prime }\right) =\left\{
\begin{array}{cc}
\left( p_{1},q_{k}^{\ast }\right) & \text{if }p_{1}\geq q_{k}^{\ast }, \\
\left( q_{k}^{\ast },p_{1}\right) & \text{if }p_{1}<q_{k}^{\ast }.%
\end{array}%
\right. \text{ \ and \ }\boldsymbol{q}_{\left[ 2\right] }^{\prime }=\left(
q_{1}^{\prime },q_{2}^{\prime }\right) =\left( q_{k},q_{k+1}\right) ,
\label{p2',q2'}
\end{equation}%
Suppose that%
\begin{equation}
q_{1}\geq q_{2}\geq \cdot \cdot \cdot \geq q_{k}\geq p_{1}\geq p_{2}\geq
\cdots \geq p_{n+1}\geq q_{n+1}.  \label{(pi,qi)-rk}
\end{equation}

If $q_{k\mathbf{+1}}\leq p_{1}$, then $\boldsymbol{p}_{\left[ n\right]
}^{\ast }\prec \boldsymbol{q}_{\left[ n\right] }^{\ast }$ and $\boldsymbol{p}%
_{\left[ 2\right] }^{\prime }\prec \boldsymbol{q}_{\left[ 2\right] }^{\prime
}$.
\end{lemma}

\begin{proof}
(i) We first show that $\boldsymbol{p}_{\left[ n\right] }^{\ast }\prec
\boldsymbol{q}_{\left[ n\right] }^{\ast }$. Since $p_{j}\geq p_{j+1}$, we
have $p_{j-1}^{\ast }\geq p_{j}^{\ast }$ for $j=2,...,n+1$. Similarly, we
have%
\begin{eqnarray*}
q_{j-1}^{\ast } &\geq &q_{j}^{\ast }\text{ for }j=2,...,k-1,k+2,...n, \\
q_{k-1}^{\ast }-q_{k}^{\ast } &=&q_{k-1}-\left( q_{k}+q_{k+1}-p_{1}\right)
=\left( q_{k-1}-q_{k}\right) +\left( p_{1}-q_{k+1}\right) \geq 0, \\
q_{k}^{\ast }-q_{k+1}^{\ast } &=&\left( q_{k}+q_{k+1}-p_{1}\right)
-q_{k+2}=\left( q_{k}-p_{1}\right) +\left( q_{k+1}-q_{k+2}\right) \geq 0%
\text{.}
\end{eqnarray*}%
Next we check that $\sum_{j=1}^{\ell }q_{j}^{\ast }\geq \sum_{j=1}^{\ell
}p_{j}^{\ast }$ for $\ell =1,...,n-1$ and $\sum_{j=1}^{n}q_{j}^{\ast }\geq
\sum_{j=1}^{n}p_{j}^{\ast }$. Since $\boldsymbol{p}_{\left[ n+1\right]
}\prec \boldsymbol{q}_{\left[ n+1\right] }$, we have%
\begin{eqnarray*}
\sum_{j=1}^{\ell }q_{j}^{\ast }-\sum_{j=1}^{\ell }p_{j}^{\ast } &=&\left(
\sum_{j=1}^{\ell }q_{j}-\sum_{j=1}^{\ell }p_{j}\right) +\left( p_{1}-p_{\ell
+1}\right) \geq 0\text{ \ for }\ell =1,...,k-1\text{,} \\
\sum_{j=1}^{k}q_{j}^{\ast }-\sum_{j=1}^{k}p_{j}^{\ast }
&=&\sum_{j=1}^{k-1}q_{j}+q_{k}+q_{k+1}-p_{1}-\sum_{j=1}^{k}p_{j+1}=%
\sum_{j=1}^{k+1}q_{j}-\sum_{j=1}^{k+1}p_{j}\geq 0\text{,} \\
\sum_{j=1}^{\ell }q_{j}^{\ast }-\sum_{j=1}^{\ell }p_{j}^{\ast }
&=&\sum_{j=1}^{k-1}q_{j}+q_{k}+q_{k+1}-p_{1}+\sum_{j=k+1}^{\ell
}q_{j+1}-\sum_{j=2}^{\ell +1}p_{j} \\
&=&\sum_{j=1}^{\ell +1}q_{j}-\sum_{j=1}^{\ell +1}p_{j}\geq 0\text{ for }\ell
=k+1,...,n\text{, }
\end{eqnarray*}%
where the last inequality becomes equality for $\ell =n$.

(ii) We now verify that $\boldsymbol{p}_{\left[ 2\right] }^{\prime }\prec
\boldsymbol{q}_{\left[ 2\right] }^{\prime }$. Clearly, $p_{1}^{\prime
}-p_{2}^{\prime }\geq 0$ and $q_{1}^{\prime }-q_{2}^{\prime
}=q_{k}-q_{k+1}\geq 0$. Also,%
\begin{eqnarray*}
q_{1}^{\prime }-p_{1}^{\prime } &=&\left\{
\begin{array}{ll}
q_{k}-p_{1}>0 & \text{if }p_{1}\geq q_{k}^{\ast }, \\
q_{k}-q_{k}^{\ast }=q_{k}-\left( q_{k}+q_{k+1}-p_{1}\right) \geq 0 & \text{%
if }p_{1}<q_{k}^{\ast },%
\end{array}%
\right. \\
q_{1}^{\prime }+q_{2}^{\prime } &=&q_{k}+q_{k+1}=p_{1}^{\prime
}+p_{2}^{\prime }.
\end{eqnarray*}%
This proves $\boldsymbol{p}_{\left[ 2\right] }^{\prime }\prec \boldsymbol{q}%
_{\left[ 2\right] }^{\prime }$, thereby completing the proof.
\end{proof}

The following lemma is an obvious fact, which will simplify the proof of
Theorem \ref{T1}.

\begin{lemma}
\label{Crit}Let $\mathbb{S}$ denote the set of all functions with the same
property $P$ and let $\oplus _{j}$ for $j=1,2,3$ be three operators. Let the
set $\mathbb{S}$ is closed for the three operations, that is, for any two
functions $A=A\left( x\right) ,B=B\left( x\right) \in \mathbb{S}$ satisfy $%
A\oplus _{j}B\in \mathbb{S}$ for $j=1,2,3$. Assume that the function $y_{%
\boldsymbol{p}_{\left[ n+1\right] },\boldsymbol{q}_{\left[ n+1\right]
}}\left( x\right) $ for $\boldsymbol{p}_{\left[ n+1\right] }\prec
\boldsymbol{q}_{\left[ n+1\right] }$ can be decomposed into the form of
\begin{equation*}
y_{\boldsymbol{p}_{\left[ n+1\right] },\boldsymbol{q}_{\left[ n+1\right]
}}\left( x\right) =\left( A\left( x\right) \oplus _{1}y_{\boldsymbol{p}_{%
\left[ 2\right] }^{\prime },\boldsymbol{q}_{\left[ 2\right] }^{\prime
}}\left( x\right) \right) \oplus _{2}\left( y_{\boldsymbol{p}_{\left[ n%
\right] }^{\ast },\boldsymbol{q}_{\left[ n\right] }^{\ast }}\left( x\right)
\oplus _{3}B\left( x\right) \right) ,
\end{equation*}%
where $\boldsymbol{p}_{\left[ 2\right] }^{\prime }\prec \boldsymbol{q}_{%
\left[ 2\right] }^{\prime }$ and $\boldsymbol{p}_{\left[ n\right] }^{\ast
}\prec \boldsymbol{q}_{\left[ n\right] }^{\ast }$. If $y_{\boldsymbol{p}_{%
\left[ 2\right] },\boldsymbol{q}_{\left[ 2\right] }}\in \mathbb{S}$ for $%
\boldsymbol{p}_{\left[ 2\right] }\prec \boldsymbol{q}_{\left[ 2\right] }$,
then $y_{\boldsymbol{p}_{\left[ n+1\right] },\boldsymbol{q}_{\left[ n+1%
\right] }}\in \mathbb{S}$ for $\boldsymbol{p}_{\left[ n+1\right] }\prec
\boldsymbol{q}_{\left[ n+1\right] }$.
\end{lemma}

\section{Proof of Theorem \protect\ref{T1}}

\begin{proof}
Let%
\begin{equation*}
\mathbb{S=}\left\{ \text{All CM functions on }\left( 0,\infty \right)
\right\} \text{.}
\end{equation*}%
(i) We first prove that $\mathcal{D}_{\boldsymbol{p}_{\left[ n\right] },%
\boldsymbol{q}_{\left[ n\right] }}\in \mathbb{S}$ if there is real $\theta
<1+\min \left\{ p_{n},q_{n}\right\} $ such that $f_{\theta }\left( t\right)
=t^{\theta }f\left( t\right) $ is log-concave on $\left( 0,\infty \right) $.
For convenience, we introduce the function $\Phi _{p}$ defined by
\begin{equation*}
\Phi _{p}\left( x\right) =\frac{F_{p}\left( x\right) }{\Gamma \left(
p-\theta +1\right) }.
\end{equation*}%
Then%
\begin{equation*}
\frac{\mathcal{D}_{\boldsymbol{p}_{\left[ n\right] },\boldsymbol{q}_{\left[ n%
\right] }}\left( x;\lambda _{n}^{\left[ \theta \right] }\right) }{%
\prod_{j=1}^{n}\Gamma \left( p_{j}-\theta +1\right) }=\prod_{j=1}^{n}\Phi
_{p_{j}}\left( x\right) -\prod_{j=1}^{n}\Phi _{q_{j}}\left( x\right)
=:\Delta _{\boldsymbol{p}_{\left[ n\right] },\boldsymbol{q}_{\left[ n\right]
}}\left( x\right) .
\end{equation*}%
Thus, it suffices to prove that $\Delta _{\boldsymbol{p}_{\left[ n\right] },%
\boldsymbol{q}_{\left[ n\right] }}\in \mathbb{S}$ if $\boldsymbol{p}%
_{[n]}\prec \boldsymbol{q}_{[n]}$.

When $n=2$, the proof of \cite[Lemma 2.2]{Yang-JMAA-551-2025} and \cite[%
Section 3.1]{Yang-JMAA-551-2025} are both valid if replacing $\boldsymbol{p}%
_{\left[ 2\right] },\boldsymbol{q}_{\left[ 2\right] }\in \mathbb{N}_{0}^{2}$
by $\boldsymbol{p}_{\left[ 2\right] },\boldsymbol{q}_{\left[ 2\right] }\in
\mathbb{I}^{2}$.

Suppose that for certain $n\geq 2$, $\Delta _{\boldsymbol{p}_{\left[ n\right]
},\boldsymbol{q}_{\left[ n\right] }}\in \mathbb{S}$ if $\boldsymbol{p}%
_{[n]}\prec \boldsymbol{q}_{[n]}$. We prove that $\Delta _{\boldsymbol{p}_{%
\left[ n+1\right] },\boldsymbol{q}_{\left[ n+1\right] }}\in \mathbb{S}$ if $%
\boldsymbol{p}_{[n+1]}\prec \boldsymbol{q}_{[n+1]}$ stepwise.

Let $p_{j}^{\ast }$ and $q_{j}^{\ast }$ are defined by (\ref{pj*,qj*}),
while $\boldsymbol{p}_{\left[ 2\right] }^{\prime }$ and $\boldsymbol{q}_{%
\left[ 2\right] }^{\prime }$ are defined by (\ref{p2',q2'}). Then $\Delta _{%
\boldsymbol{p}_{\left[ n+1\right] },\boldsymbol{q}_{\left[ n+1\right] }}$
can be written in the form of%
\begin{eqnarray}
\Delta _{\boldsymbol{p}_{\left[ n+1\right] },\boldsymbol{q}_{\left[ n+1%
\right] }} &=&\left( \prod_{j=1}^{n+1}\Phi _{p_{j}}-\Phi
_{p_{1}}\prod_{j=1}^{n}\Phi _{q_{j}^{\ast }}\right) +\left( \Phi
_{p_{1}}\prod_{j=1}^{n}\Phi _{q_{j}^{\ast }}-\prod_{j=1}^{n+1}\Phi
_{q_{j}}\right)  \notag \\
&=&\Phi _{p_{1}}\Delta _{\boldsymbol{p}_{\left[ n\right] }^{\ast },%
\boldsymbol{q}_{\left[ n\right] }^{\ast }}+\left( \prod_{j=k+2}^{n+1}\Phi
_{q_{j}}\right) \left( \prod_{j=1}^{k-1}\Phi _{q_{j}}\right) \Delta _{%
\boldsymbol{p}_{\left[ 2\right] }^{\prime },\boldsymbol{q}_{\left[ 2\right]
}^{\prime }},  \label{Dp,qn+1-Dp,qn,2}
\end{eqnarray}%
where%
\begin{equation*}
\Delta _{\boldsymbol{p}_{\left[ n\right] }^{\ast },\boldsymbol{q}_{\left[ n%
\right] }^{\ast }}=\prod_{j=1}^{n}\Phi _{p_{j}^{\ast }}-\prod_{j=1}^{n}\Phi
_{q_{j}^{\ast }}\text{ \ and \ }\Delta _{\boldsymbol{p}_{\left[ 2\right]
}^{\prime },\boldsymbol{q}_{\left[ 2\right] }^{\prime }}=\Phi _{p_{1}}\Phi
_{q_{k}^{\ast }}-\Phi _{q_{k}}\Phi _{q_{k+1}}.
\end{equation*}%
Clearly, the operators in the right hand side of (\ref{Dp,qn+1-Dp,qn,2}) are
just ordinary multiplication and addition operations, which satisfies $%
A\times B,A+B\in \mathbb{S}$ if $A,B\in \mathbb{S}$.

We now begin with the inequalities%
\begin{equation}
q_{1}\geq p_{1}\geq p_{2}\geq \cdots \geq p_{n+1}\geq q_{n+1}.
\label{(pi,qi)-r1}
\end{equation}%
which follow from the known condition that $\boldsymbol{p}_{\left[ n+1\right]
}\prec \boldsymbol{q}_{\left[ n+1\right] }$.

\textbf{Step 1}: Let $p_{j}^{\ast }=p_{j+1}$ and%
\begin{equation*}
q_{1}^{\ast }=q_{1}+q_{2}-p_{1}\text{ \ and \ }q_{j}^{\ast }=q_{j+1}\text{
for }j=2,...,n+1\text{.}
\end{equation*}

Case 1.1: $q_{2}\leq p_{1}$. Taking $k=1$ in Lemma \ref{L-p,qn*,p,q2'}
reveals that $\boldsymbol{p}_{\left[ n\right] }^{\ast }\prec \boldsymbol{q}_{%
\left[ n\right] }^{\ast }$ and $\boldsymbol{p}_{\left[ 2\right] }^{\prime
}\prec \boldsymbol{q}_{\left[ 2\right] }^{\prime }$. By the base case and
induction assumption, $\Delta _{\boldsymbol{p}_{\left[ 2\right] }^{\prime },%
\boldsymbol{q}_{\left[ 2\right] }^{\prime }},\Delta _{\boldsymbol{p}_{\left[
n\right] }^{\ast },\boldsymbol{q}_{\left[ n\right] }^{\ast }}\in \mathbb{S}$%
, which, by Lemma \ref{Crit}, yields $\Delta _{\boldsymbol{p}_{\left[ n+1%
\right] },\boldsymbol{q}_{\left[ n+1\right] }}\in \mathbb{S}$.

Case 1.2: $q_{2}>p_{1}$. This in combination with (\ref{(pi,qi)-r1}) leads to%
\begin{equation}
q_{1}\geq q_{2}>p_{1}\geq p_{2}\geq \cdots \geq p_{n+1}\geq q_{n+1}.
\label{(pi,qi)-r2}
\end{equation}%
To prove $\Delta _{\boldsymbol{p}_{\left[ n+1\right] },\boldsymbol{q}_{\left[
n+1\right] }}\in \mathbb{S}$ in Case 1.2, it suffices to prove $\Delta _{%
\boldsymbol{p}_{\left[ n+1\right] },\boldsymbol{q}_{\left[ n+1\right] }}\in
\mathbb{S}$ under the condition (\ref{(pi,qi)-r2}), which can be proved by
Step 2.

\textbf{Step 2}: Let $p_{j}^{\ast }=p_{j+1}$ and%
\begin{equation*}
q_{1}^{\ast }=q_{1}\text{, \ }q_{2}^{\ast }=q_{2}+q_{3}-p_{1}\text{ \ and \ }%
q_{j}^{\ast }=q_{j+1}\text{ for }j=3,...,n+1\text{.}
\end{equation*}

Case 2.1: $q_{3}\leq p_{1}$. Taking $k=2$ in Lemma \ref{L-p,qn*,p,q2'}
reveals that $\boldsymbol{p}_{\left[ n\right] }^{\ast }\prec \boldsymbol{q}_{%
\left[ n\right] }^{\ast }$ and $\boldsymbol{p}_{\left[ 2\right] }^{\prime
}\prec \boldsymbol{q}_{\left[ 2\right] }^{\prime }$. By the base case and
induction assumption, $\Delta _{\boldsymbol{p}_{\left[ 2\right] }^{\prime },%
\boldsymbol{q}_{\left[ 2\right] }^{\prime }},\Delta _{\boldsymbol{p}_{\left[
n\right] }^{\ast },\boldsymbol{q}_{\left[ n\right] }^{\ast }}\in \mathbb{S}$%
, which, by Lemma \ref{Crit}, yields $\Delta _{\boldsymbol{p}_{\left[ n+1%
\right] },\boldsymbol{q}_{\left[ n+1\right] }}\in \mathbb{S}$.

Case 2.2: $q_{3}>p_{1}$. This in combination with (\ref{(pi,qi)-r2}) leads to%
\begin{equation}
q_{1}\geq q_{2}\geq q_{3}>p_{1}\geq p_{2}\geq \cdots \geq p_{n+1}\geq
q_{n+1}.  \label{(pi,qi)-r3}
\end{equation}%
To prove $\Delta _{\boldsymbol{p}_{\left[ n+1\right] },\boldsymbol{q}_{\left[
n+1\right] }}\in \mathbb{S}$ in Case 2.2, it suffices to prove that $\Delta
_{\boldsymbol{p}_{\left[ n+1\right] },\boldsymbol{q}_{\left[ n+1\right]
}}\in \mathbb{S}$ under the condition (\ref{(pi,qi)-r3}), which can be
proved by Step 3.

......

Repeating such step $n-1$ times, it remains to prove that $\Delta _{%
\boldsymbol{p}_{\left[ n+1\right] },\boldsymbol{q}_{\left[ n+1\right] }}\in
\mathbb{S}$ under the condition that%
\begin{equation}
q_{1}\geq q_{2}\geq \cdots \geq q_{n-1}\geq q_{n}>p_{1}\geq p_{2}\geq \cdots
\geq p_{n+1}\geq q_{n+1}.  \label{(pi,qi)-rn}
\end{equation}

\textbf{Step n}: Let $p_{j}^{\ast }=p_{j+1}$ and%
\begin{equation*}
q_{j}^{\ast }=q_{j}\text{ for }j=1,,,,,n-1\text{ \ and \ }q_{n}^{\ast
}=q_{n}+q_{n+1}-p_{1}\text{.}
\end{equation*}

Taking $k=n$ in Lemma \ref{L-p,qn*,p,q2'} reveals that $\boldsymbol{p}_{%
\left[ n\right] }^{\ast }\prec \boldsymbol{q}_{\left[ n\right] }^{\ast }$
and $\boldsymbol{p}_{\left[ 2\right] }^{\prime }\prec \boldsymbol{q}_{\left[
2\right] }^{\prime }$. By the base case and induction assumption, $\Delta _{%
\boldsymbol{p}_{\left[ 2\right] }^{\prime },\boldsymbol{q}_{\left[ 2\right]
}^{\prime }},\Delta _{\boldsymbol{p}_{\left[ n\right] }^{\ast },\boldsymbol{q%
}_{\left[ n\right] }^{\ast }}\in \mathbb{S}$, which, by Lemma \ref{Crit},
yields $\Delta _{\boldsymbol{p}_{\left[ n+1\right] },\boldsymbol{q}_{\left[
n+1\right] }}\in \mathbb{S}$.

Taking into account the above $n$ times steps, we deduce that $\Delta _{%
\boldsymbol{p}_{\left[ n+1\right] },\boldsymbol{q}_{\left[ n+1\right] }}\in
\mathbb{S}$ when $\boldsymbol{p}_{\left[ n+1\right] }\prec \boldsymbol{q}_{%
\left[ n+1\right] }$. By induction, the function $\Delta _{\boldsymbol{p}_{%
\left[ n\right] },\boldsymbol{q}_{\left[ n\right] }}\in \mathbb{S}$ when $%
\boldsymbol{p}_{\left[ n\right] }\prec \boldsymbol{q}_{\left[ n\right] }$
for all $n\geq 2$.

(ii) If there is real $\theta <1+\min \left\{ p_{n},q_{n}\right\} $ such
that $f_{\theta }\left( t\right) =t^{\theta }f\left( t\right) $ is
log-convex on $\left( 0,\infty \right) $, we can prove that $-\Delta _{%
\boldsymbol{p}_{\left[ n\right] },\boldsymbol{q}_{\left[ n\right] }}\in
\mathbb{S}$ when $\boldsymbol{p}_{\left[ n\right] }\prec \boldsymbol{q}_{%
\left[ n\right] }$ for all $n\geq 2$ in the same way as part (i).

This completes the proof.
\end{proof}

\section{Applications to special functions}

As applications of \cite[Theorem 1.1]{Yang-JMAA-551-2025}, the author
obtained three complete monotonicity results involving psi function,
Nielsen's beta function and Tricomi function; that are, \cite[Propositions
4.1, 5.1, 6.1]{Yang-JMAA-551-2025}. On applying our extended Theorem \ref{T1}%
, these results can also be extended. In additional, an application to the
Mills ratio is presented in this section.

\subsection{Application to the Hurwitz zeta function: Extension of
\protect\cite[Proposition 4.1]{Yang-JMAA-551-2025}}

Consider the function%
\begin{equation}
\psi _{p}\left( x\right) =\int_{0}^{\infty }\frac{t^{p}}{1-e^{-t}}e^{-xt}dt,
\label{Psip1}
\end{equation}%
which is the polygamma function if $p\in \mathbb{N}$; that is, $\psi
_{p}\left( x\right) =\left( -1\right) ^{p-1}\psi ^{\left( p\right) }\left(
x\right) $ if $p\in \mathbb{N}$. For this reason, $\psi _{p}$ can also be
regarded as \textquotedblleft extended polygamma function\textquotedblright\
by F. Qi.

Now, series expansion gives%
\begin{equation}
\begin{array}{cc}
\psi _{p}\left( x\right) & =\int_{0}^{\infty }t^{p}\left(
\sum\limits_{k=0}^{\infty }e^{-kt}\right)
e^{-xt}dt=\sum\limits_{k=0}^{\infty }\left( \int_{0}^{\infty
}t^{p}e^{-\left( k+x\right) t}dt\right) \\
& =\sum\limits_{k=0}^{\infty }\frac{\Gamma \left( p+1\right) }{\left(
k+x\right) ^{p+1}}=\Gamma \left( p+1\right) \zeta \left( p+1,x\right) ,%
\end{array}
\label{Psip-1-Hz}
\end{equation}%
where%
\begin{equation}
\zeta \left( p,x\right) =\sum_{k=0}^{\infty }\frac{1}{\left( k+x\right) ^{p}}
\label{Hzf}
\end{equation}%
is the Hurwitz zeta function (or generalized zeta function) \cite[p. 25]%
{Magnus-FTSFMP-S-1966}. Since $\zeta \left( p,x\right) $ converges in $x$ on
$\left( 0,\infty \right) $ for $p>1$, so is $\psi _{p}\left( x\right) $ for $%
p>0$.

The Hurwitz zeta function $\zeta \left( p,x\right) $ has an asymptotic
expansion \cite[p. 25]{Magnus-FTSFMP-S-1966}%
\begin{equation}
\zeta \left( p,x\right) \sim \frac{1}{\left( p-1\right) x^{p-1}}+\frac{1}{%
2x^{p}}+\sum_{n=1}^{\infty }\frac{B_{2n}}{\left( 2n\right) !}\frac{\Gamma
\left( 2n+p-1\right) }{\Gamma \left( p\right) }\frac{1}{x^{2n-1+p}}\text{ as
}x\rightarrow \infty ,  \label{Hzf-ae}
\end{equation}%
where $B_{2n}$ for $n\in \mathbb{N}$ are the Bernoulli numbers defined by%
\begin{equation*}
\frac{t}{e^{t}-1}=\sum_{n=0}^{\infty }B_{n}\frac{t^{n}}{n!}\text{ \ \ }%
\left\vert t\right\vert <2\pi .
\end{equation*}%
By (\ref{Hzf}) and (\ref{Hzf-ae}), we have%
\begin{equation}
\lim_{x\rightarrow 0^{+}}\left( x^{p}\zeta \left( p,x\right) \right) =1\text{
\ and \ }\lim_{x\rightarrow \infty }\left( x^{p-1}\zeta \left( p,x\right)
\right) =\frac{1}{p-1}.  \label{zp-af}
\end{equation}

As shown in \cite[p. 9]{Yang-JMAA-551-2025}, the function $f_{\theta }\left(
t\right) =t^{\theta }f\left( t\right) =t^{\theta +1}/\left( 1-e^{-t}\right) $
satisfies that%
\begin{equation*}
\left[ \ln f_{\theta }\left( t\right) \right] ^{\prime \prime }=-\frac{1}{%
t^{2}}\left( \theta +1-\frac{t^{2}e^{t}}{\left( e^{t}-1\right) ^{2}}\right)
\leq \left( \geq \right) 0
\end{equation*}%
for $t>0$ if and only if $\theta \geq 0$ ($\theta \leq -1$). Thus, by
Theorem \ref{T1}, \cite[Proposition 4.1]{Yang-JMAA-551-2025} can be extended
as follows.

\begin{proposition}
\label{P-gTd-cm-ep}Let the $n$-tuples $\boldsymbol{p}_{\left[ n\right]
}=\left( p_{1},...,p_{n}\right) $ and $\boldsymbol{q}_{\left[ n\right]
}=\left( q_{1},...,q_{n}\right) \in \left( 0,\infty \right) ^{n}$ for $n\geq
2$ satisfy $\boldsymbol{p}_{\left[ n\right] }\prec \boldsymbol{q}_{\left[ n%
\right] }$. Then the functions%
\begin{eqnarray*}
x &\mapsto &\prod_{j=1}^{n}\psi _{p_{j}}\left( x\right) -\lambda _{n}^{\left[
1\right] }\prod_{j=1}^{n}\psi _{q_{j}}\left( x\right) , \\
x &\mapsto &-\prod_{j=1}^{n}\psi _{p_{j}}\left( x\right) +\lambda _{n}^{
\left[ 0\right] }\prod_{j=1}^{n}\psi _{q_{j}}\left( x\right)
\end{eqnarray*}%
are CM on $\left( 0,\infty \right) $ with the best constants $\lambda _{n}^{
\left[ 1\right] }$ and $\lambda _{n}^{\left[ 0\right] }$, where $\lambda
_{n}^{\left[ \theta \right] }$ is defined by (\ref{ln-thi}). are the best
constants. Consequently, the double inequality%
\begin{equation}
\prod_{j=1}^{n}\frac{\Gamma \left( p_{j}\right) }{\Gamma \left( q_{j}\right)
}<\prod_{j=1}^{n}\frac{\psi _{p_{j}}\left( x\right) }{\psi _{q_{j}}\left(
x\right) }<\prod_{j=1}^{n}\frac{\Gamma \left( p_{j}+1\right) }{\Gamma \left(
q_{j}+1\right) }  \label{Ppi/qi<>}
\end{equation}%
holds for all $x>0$. The lower and upper bounds are sharp.
\end{proposition}

\begin{proof}
The required complete monotonicity follows from Theorem \ref{T1}. It remains
to prove that $\lambda _{k}^{\left[ 1\right] }$ and $\lambda _{k}^{\left[ 0%
\right] }$ are the best constants. The asymptotic formulas (\ref{zp-af})
with the relation (\ref{Psip-1-Hz}) yields that%
\begin{equation*}
\prod_{j=1}^{n}\frac{\psi _{p_{j}}\left( x\right) }{\psi _{q_{j}}\left(
x\right) }\thicksim \left\{
\begin{array}{ll}
\prod\limits_{j=1}^{n}\frac{\Gamma \left( p_{j}+1\right) }{\Gamma \left(
q_{j}+1\right) }=\lambda _{k}^{\left[ 0\right] } & \text{as }x\rightarrow
0^{+}, \\
\prod\limits_{j=1}^{n}\frac{\Gamma \left( p_{j}\right) }{\Gamma \left(
q_{j}\right) }=\lambda _{k}^{\left[ 1\right] }, & \text{as }x\rightarrow
\infty ,%
\end{array}%
\right.
\end{equation*}%
which implies that $\lambda _{k}^{\left[ 1\right] }$ and $\lambda _{k}^{%
\left[ 0\right] }$ are the best constants. Accordingly, the lower and upper
bounds in (\ref{Ppi/qi<>}) are sharp. This completes the proof.
\end{proof}

Using the relation%
\begin{equation}
\psi _{p-1}\left( x\right) =\Gamma \left( p\right) \zeta \left( p,x\right)
\text{ for }x>0\text{ and }p>1,  \label{psip-1-Hz}
\end{equation}%
Proposition \ref{P-gTd-cm-ep} can be equivalently stated as follows.

\begin{proposition}
\label{P-gTd-cm-Hz}Let the $n$-tuples $\boldsymbol{p}_{\left[ n\right]
}=\left( p_{1},...,p_{n}\right) $ and $\boldsymbol{q}_{\left[ n\right]
}=\left( q_{1},...,q_{n}\right) \in \left( 1,\infty \right) ^{n}$ for $n\geq
2$ satisfy $\boldsymbol{p}_{\left[ n\right] }\prec \boldsymbol{q}_{\left[ n%
\right] }$. Then the function%
\begin{equation*}
x\mapsto \mathcal{D}_{\boldsymbol{p}_{\left[ n\right] },\boldsymbol{q}_{%
\left[ n\right] }}^{\left[ H\right] }\left( x;\lambda _{n}\right)
=\prod_{j=1}^{n}\zeta \left( p_{j},x\right) -\lambda
_{n}\prod_{j=1}^{n}\zeta \left( q_{j},x\right)
\end{equation*}%
is CM on $\left( 0,\infty \right) $ if and only if%
\begin{equation*}
\lambda _{n}\leq \prod_{j=1}^{n}\frac{q_{j}-1}{p_{j}-1};
\end{equation*}%
while the function $-\mathcal{D}_{\boldsymbol{p}_{\left[ n\right] },%
\boldsymbol{q}_{\left[ n\right] }}^{\left[ H\right] }\left( x;\lambda
_{n}\right) $ is CM in $x$ on $\left( 0,\infty \right) $ if and only if and $%
\lambda _{n}\geq 1$. Consequently, the double inequality%
\begin{equation}
\prod_{j=1}^{n}\frac{q_{j}-1}{p_{j}-1}<\prod_{j=1}^{n}\frac{\zeta \left(
p_{j},x\right) }{\zeta \left( q_{j},x\right) }<1  \label{Pzpj/zqj<>}
\end{equation}%
holds for all $x>0$. The lower and upper bounds are sharp.
\end{proposition}

Since $\left( \left( p+q\right) /2,\left( p+q\right) /2\right) \prec \left(
p,q\right) $ if $p\geq q>1$ and $\left( \left( p+q\right) /2,\left(
p+q\right) /2\right) \prec \left( q,p\right) $ if $q<p>1$, by the
inequalities (\ref{Pzpj/zqj<>}), we have the following Tur\'{a}n type
inequalities%
\begin{equation}
\frac{\left( p-1\right) \left( q-1\right) }{\left( \left( p+q\right)
/2-1\right) ^{2}}<\frac{\zeta \left( \left( p+q\right) /2,x\right) ^{2}}{%
\zeta \left( p,x\right) \zeta \left( q,x\right) }<1  \label{Pzpj/zqj<>s}
\end{equation}%
for all $x>0$ and $p,q>1$.

\begin{remark}
The second inequality of (\ref{Pzpj/zqj<>s}) is superior to the Tur\'{a}n
type inequality \cite[Eq. (2.23)]{Kumar-ITSF-27-2016}.
\end{remark}

The second and first inequality imply the log-convexity of $\zeta \left(
p,x\right) $ and log-concavity of $\left( p-1\right) \zeta \left( p,x\right)
$ in $p$ on $\left( 1,\infty \right) $.

\begin{corollary}
\label{C-P1}For fixed $x>0$, the functions $p\mapsto \zeta \left( p,x\right)
$ and $p\mapsto \left( p-1\right) \zeta \left( p,x\right) $ are log-convex
and log-concave on $\left( 1,\infty \right) $, respectively.
\end{corollary}

\begin{remark}
\label{R-CP1}The second inequality of (\ref{Pzpj/zqj<>s}) is equivalent to%
\begin{equation*}
x^{ap+b}\zeta \left( p,x\right) x^{aq+b}\zeta \left( q,x\right) >\left(
x^{a\left( p+q\right) /2+b}\zeta \left( \frac{p+q}{2},x\right) \right) ^{2},
\end{equation*}%
which implies that $p\mapsto x^{ap+b}\zeta \left( p,x\right) $ is log-convex
on $\left( 1,\infty \right) $ for any $a,b\in \mathbb{R}$. Putting $\left(
a,b\right) =\left( 1,0\right) $ yields \cite[Theorem (ii)]%
{Cvijovic-JMAA-487-2020}. This provides a new proof of \cite[Theorem (ii)]%
{Cvijovic-JMAA-487-2020}. Similarly, the function $p\mapsto \left(
p-1\right) x^{ap+b}\zeta \left( p,x\right) $ is also log-concave on $\left(
1,\infty \right) $ for any $a,b\in \mathbb{R}$.
\end{remark}

The Dirichlet's lambda function (see \cite[Ch. 3]{Oldham-AAF-SVNY-2009}) is
defined by%
\begin{equation*}
\lambda \left( p\right) =\sum_{k=1}^{\infty }\frac{1}{\left( 2k-1\right) ^{p}%
}=\frac{1}{2^{p}}\zeta \left( p,\frac{1}{2}\right) =\left( 1-2^{-p}\right)
\zeta \left( p\right) ,
\end{equation*}%
for $p>1$, where $\zeta \left( p\right) $ is the Riemann zeta function. Now,
by Corollary \ref{C-P1} and Remark \ref{R-CP1}, the function $\lambda \left(
p\right) $ and $\left( p-1\right) \lambda \left( p\right) $ are log-convex
and log-concave on $\left( 1,\infty \right) $, respectively. Then for $r>0$,
the functions%
\begin{equation*}
\frac{\lambda \left( p+r\right) }{\lambda \left( p\right) }\text{ \ and \ }%
\frac{\left( p+r-1\right) \lambda \left( p+r+1\right) }{\left( p-1\right)
\lambda \left( p\right) }
\end{equation*}%
are increasing and decreasing on $\left( 1,\infty \right) $ with%
\begin{equation*}
\lim_{p\rightarrow \infty }\frac{\lambda \left( p+r\right) }{\lambda \left(
p\right) }=\lim_{p\rightarrow \infty }\frac{\left( p+r-1\right) \lambda
\left( p+r+1\right) }{\left( p-1\right) \lambda \left( p\right) }=1.
\end{equation*}%
Therefore, the double inequality%
\begin{equation*}
\frac{p-1}{p+r-1}<\frac{\lambda \left( p+r\right) }{\lambda \left( p\right) }%
<1
\end{equation*}%
hold. By the relation $\lambda \left( p\right) =\left( 1-2^{-p}\right) \zeta
\left( p\right) $, we obtain the bounds for the ratio of two Riemann zeta
functions, which is a complement to \cite[Corollary 1]{Yang-JCAM-364-2020}.

\begin{corollary}
Let $p>1$ and $r>0$. The double inequality
\begin{equation}
\frac{p-1}{p+r-1}\frac{\left( 2^{p}-1\right) 2^{r}}{2^{p+r}-1}<\frac{\zeta
\left( p+r\right) }{\zeta \left( p\right) }<\frac{\left( 2^{p}-1\right) 2^{r}%
}{2^{p+r}-1}  \label{zp+r/zp<>}
\end{equation}%
holds.
\end{corollary}

\begin{remark}
Taking $\left( p,r\right) =\left( 2n,2\right) $ for $n\in \mathbb{N}$ in (%
\ref{zp+r/zp<>}) and using the relation%
\begin{equation*}
\zeta \left( 2n\right) =\frac{\left( 2\pi \right) ^{2n}}{2\left( 2n\right) !}%
\left\vert B_{2n}\right\vert \text{ \ (\cite[Eq. (23.2.16)]%
{Abramowitz-HMFFGMT-1972})}
\end{equation*}%
indicate the double inequality%
\begin{equation}
\frac{\left( 2n+2\right) \left( 2n-1\right) }{\pi ^{2}}\frac{2^{2n}-1}{%
2^{2n+2}-1}<\frac{\left\vert B_{2n+2}\right\vert }{\left\vert
B_{2n}\right\vert }<\frac{\left( 2n+2\right) \left( 2n+1\right) }{\pi ^{2}}%
\frac{2^{2n}-1}{2^{2n+2}-1}  \label{B2n+2/B2n<>}
\end{equation}%
for $n\in \mathbb{N}$. The right hand side inequality of (\ref{B2n+2/B2n<>})
was first proved by Qi \cite[Theorem 1.1]{Qi-JCAM-351-2019}, and the left
hand side one is weaker than \cite[Eq. (1.3)]{Qi-JCAM-351-2019}.
\end{remark}

We close this subsection by giving an inequality which is near to answer to
an open problem in \cite[Open problem 2]{Yang-JCAM-364-2020}.

\begin{corollary}
Let $p>1$ and $q>0$. The inequality%
\begin{equation}
\frac{1}{\zeta \left( p+q\right) }<\frac{\alpha _{p,q}}{\zeta \left(
p+2q\right) }+\frac{\beta _{p,q}}{\zeta \left( p\right) }  \label{1/zp+q<}
\end{equation}%
hold, where%
\begin{equation*}
\alpha _{p,q}=\frac{p+q-1}{p-1}\frac{2^{p+q}-1}{2^{q+1}\left( 2^{p}-1\right)
}\text{ and }\beta _{p,q}=\frac{p+q-1}{p+2q-1}\frac{2^{q-1}\left(
2^{p+q}-1\right) }{2^{p+2q}-1}
\end{equation*}
\end{corollary}

\begin{proof}
The log-concavity of the function $\left( p-1\right) \lambda \left( p\right)
$ on $\left( 1,\infty \right) $ implies that the function $1/\left( \left(
p-1\right) \lambda \left( p\right) \right) $ is convex on $\left( 1,\infty
\right) $. Hence, the inequality%
\begin{equation*}
\frac{1}{\left( p+q-1\right) \lambda \left( p+q\right) }<\frac{1/2}{\left(
p-1\right) \lambda \left( p\right) }+\frac{1/2}{\left( p+2q-1\right) \lambda
\left( p+2q\right) }
\end{equation*}%
holds for $p>1$ and $q>0$. Substituting $\lambda \left( p\right) =\left(
1-2^{-p}\right) \zeta \left( p\right) $ into the above inequality then
arranging leads to (\ref{1/zp+q<}).
\end{proof}

\subsection{Application to the alternating Hurwitz zeta function: Extension
of \protect\cite[Proposition 5.1]{Yang-JMAA-551-2025}}

Consider the function%
\begin{equation*}
\beta _{p}\left( x\right) =\int_{0}^{\infty }\frac{t^{p}e^{-tx}}{1+e^{-t}}dt,
\end{equation*}%
which satisfies $\beta _{p}\left( x\right) =\left( -1\right) ^{p-1}\beta
^{\left( p\right) }\left( x\right) $ if $p\in \mathbb{N}$, where%
\begin{equation*}
\beta _{0}\left( x\right) =\beta \left( x\right) =\int_{0}^{\infty }\frac{%
e^{-tx}}{1+e^{-t}}dt=\psi \left( x\right) -\psi \left( \frac{x}{2}\right)
-\ln 2
\end{equation*}%
is called Nielsen's $\beta $-function (see \cite{Nielsen-HTG-1906}, \cite%
{Nantomah-PAIA-6-2017}).

Now, series expansion gives%
\begin{equation}
\begin{array}{cc}
\beta _{p}\left( x\right) & =\int_{0}^{\infty }t^{p}\left(
\sum\limits_{k=0}^{\infty }\left( -1\right) ^{k}e^{-kt}\right)
e^{-xt}dt=\sum\limits_{k=0}^{\infty }\left( -1\right) ^{k}\left(
\int_{0}^{\infty }t^{p}e^{-\left( k+x\right) t}dt\right) \\
& =\Gamma \left( p+1\right) \sum\limits_{k=0}^{\infty }\frac{\left(
-1\right) ^{k}}{\left( k+x\right) ^{p+1}}=\Gamma \left( p+1\right) \zeta
^{\ast }\left( p+1,x\right) ,%
\end{array}
\label{bp-aHzp+1}
\end{equation}%
where%
\begin{equation}
\zeta ^{\ast }\left( p,x\right) =\sum_{k=0}^{\infty }\frac{\left( -1\right)
^{k}}{\left( k+x\right) ^{p}}  \label{aHzf}
\end{equation}%
is the alternating Hurwitz zeta function. Since $\zeta ^{\ast }\left(
p,x\right) $ converges on $\left( 0,\infty \right) $ for $p>0$, so is $\beta
_{p}\left( x\right) $ for $p>-1$.

The alternating Hurwitz zeta function $\zeta ^{\ast }\left( p,x\right) $ has
the asymptotic expansion \cite[Theorem 3.1]{Hu-JMAA-537-2024}\ that%
\begin{equation}
\zeta ^{\ast }\left( p,x\right) \sim \frac{1}{2x^{p}}+\frac{p}{4x^{p+1}}-%
\frac{1}{2}\sum_{k=1}^{\infty }\frac{E_{2k+1}\left( 0\right) }{\left(
2k+1\right) !}\frac{\left( p\right) _{2k+1}}{x^{2k+1+p}}\text{ as }%
x\rightarrow \infty ,  \label{aHz-ae}
\end{equation}%
where $E_{2k+1}\left( 0\right) $ are the special values of odd-order Euler
polynomials at $0$. This and (\ref{aHzf}) indicate that%
\begin{equation}
\lim_{x\rightarrow 0^{+}}\left( x^{p}\zeta ^{\ast }\left( p,x\right) \right)
=1\text{ \ and \ }\lim_{x\rightarrow \infty }\left( x^{p}\zeta ^{\ast
}\left( p,x\right) \right) =\frac{1}{2}.  \label{aHz-af}
\end{equation}

Moreover, it was shown in \cite[p. 10]{Yang-JMAA-551-2025} that, the
function $f_{\theta }\left( t\right) =t^{\theta }f\left( t\right) =t^{\theta
}/\left( 1+e^{-t}\right) $ satisfies that $\left[ \ln f_{\theta }\left(
t\right) \right] ^{\prime \prime }\geq \left( \leq \right) 0$ for $t>0$ if
and only if $\theta \leq -\theta _{0}$ ($\geq 0$), where%
\begin{equation*}
\theta _{0}=\frac{t_{0}^{2}e^{t_{0}}}{\left( e^{t_{0}}+1\right) ^{2}}%
=-0.439\,228...\text{,}
\end{equation*}%
and $t_{0}=2.399\,357...$ is the unique solution of the equation $%
te^{t}-2e^{t}-t-2=0$ on $\left( 0,\infty \right) $. Thus, \cite[Proposition
5.1]{Yang-JMAA-551-2025} can be extended as follows.

\begin{proposition}
\label{P-gTd-cm-b}Let the $n$-tuples $\boldsymbol{p}_{\left[ n\right]
}=\left( p_{1},...,p_{n}\right) $ and $\boldsymbol{q}_{\left[ n\right]
}=\left( q_{1},...,q_{n}\right) \in \left( -1,\infty \right) ^{n}$ for $%
n\geq 2$ satisfy $\boldsymbol{p}_{\left[ n\right] }\prec \boldsymbol{q}_{%
\left[ n\right] }$. If $0\leq \theta <1+\min \left\{ p_{n},q_{n}\right\} $,
then the function%
\begin{equation*}
x\mapsto \mathcal{D}_{\boldsymbol{p}_{\left[ n\right] },\boldsymbol{q}_{%
\left[ n\right] }}^{\left[ \beta \right] }\left( x;\lambda _{n}^{\left[
\theta \right] }\right) =\prod_{j=1}^{n}\beta _{p_{j}}\left( x\right)
-\lambda _{n}^{\left[ \theta \right] }\prod_{j=1}^{n}\beta _{q_{j}}\left(
x\right)
\end{equation*}%
is CM on $\left( 0,\infty \right) $, where $\lambda _{n}^{\left[ \theta %
\right] }$ is defined by (\ref{ln-thi}). If $\theta \leq \theta
_{0}=-0.439\,228...$, then $-\mathcal{D}_{\boldsymbol{p}_{\left[ n\right] },%
\boldsymbol{q}_{\left[ n\right] }}$ is CM in $x$ on $\left( 0,\infty \right)
$. Consequently, if $0\leq \vartheta <1+\min \left\{ p_{n},q_{n}\right\} $
and $\theta \leq \theta _{0}$, then the double inequality%
\begin{equation}
\prod_{j=1}^{n}\frac{\Gamma \left( p_{j}-\vartheta +1\right) }{\Gamma \left(
q_{j}-\vartheta +1\right) }<\frac{\prod_{j=1}^{n}\beta _{p_{j}}\left(
x\right) }{\prod_{j=1}^{n}\beta _{q_{j}}\left( x\right) }<\prod_{j=1}^{n}%
\frac{\Gamma \left( p_{j}-\theta +1\right) }{\Gamma \left( q_{j}-\theta
+1\right) }  \label{Pbpi/bqi<>}
\end{equation}%
holds for all $x>0$, where the lower bound for $\vartheta =0$ is sharp.
\end{proposition}

\begin{proof}
The required complete monotonicity follows from Theorem \ref{T1}. It remains
to prove the lower bound in (\ref{Pbpi/bqi<>}) is sharp. The asymptotic
formulas (\ref{aHz-af}) with the relation (\ref{bp-aHzp+1}) gives for $%
\vartheta =0$%
\begin{equation*}
\lim_{x\rightarrow 0^{+}}\frac{\prod_{j=1}^{n}\beta _{p_{j}}\left( x\right)
}{\prod_{j=1}^{n}\beta _{q_{j}}\left( x\right) }=\lim_{x\rightarrow \infty }%
\frac{\prod_{j=1}^{n}\beta _{p_{j}}\left( x\right) }{\prod_{j=1}^{n}\beta
_{q_{j}}\left( x\right) }=\prod_{j=1}^{n}\frac{\Gamma \left( p_{j}+1\right)
}{\Gamma \left( q_{j}+1\right) },
\end{equation*}%
which implies that the lower bound in (\ref{Pbpi/bqi<>}) for $\vartheta =0$
is sharp. This completes the proof.
\end{proof}

By the relation%
\begin{equation}
\beta _{p-1}\left( x\right) =\Gamma \left( p\right) \zeta ^{\ast }\left(
p,x\right) \text{ for }x>0\text{ and }p>0,  \label{bp-1-aHz}
\end{equation}%
Proposition \ref{P-gTd-cm-b} can be equivalently stated as follows.

\begin{proposition}
\label{P-gTd-cm-aHz}Let the $n$-tuples $\boldsymbol{p}_{\left[ n\right]
}=\left( p_{1},...,p_{n}\right) $ and $\boldsymbol{q}_{\left[ n\right]
}=\left( q_{1},...,q_{n}\right) \in \left( 0,\infty \right) ^{n}$ for $n\geq
2$ satisfy $\boldsymbol{p}_{\left[ n\right] }\prec \boldsymbol{q}_{\left[ n%
\right] }$. If $0\leq \theta <\min \left\{ p_{n},q_{n}\right\} $, then the
function%
\begin{equation*}
x\mapsto \mathcal{D}_{\boldsymbol{p}_{\left[ n\right] },\boldsymbol{q}_{%
\left[ n\right] }}^{\left[ aH\right] }\left( x;\lambda _{n}^{\left[ \theta %
\right] }\right) =\prod_{j=1}^{n}\zeta ^{\ast }\left( p_{j},x\right) -\frac{%
\lambda _{n}^{\left[ \theta +1\right] }}{\lambda _{n}^{\left[ 1\right] }}%
\prod_{j=1}^{n}\zeta ^{\ast }\left( q_{j},x\right)
\end{equation*}%
is CM on $\left( 0,\infty \right) $, where $\lambda _{n}^{\left[ \theta %
\right] }$ is defined by (\ref{ln-thi}); if $\theta \leq \theta
_{0}=-0.439\,228...$, then $-\mathcal{D}_{\boldsymbol{p}_{\left[ n\right] },%
\boldsymbol{q}_{\left[ n\right] }}^{\left[ aH\right] }$ is CM in $x$ on $%
\left( 0,\infty \right) $. Consequently, if $0\leq \vartheta <\min \left\{
p_{n},q_{n}\right\} $ and $\theta \leq \theta _{0}$, then the double
inequality%
\begin{equation}
\prod_{j=1}^{n}\frac{\Gamma \left( p_{j}-\vartheta \right) \Gamma \left(
q_{j}\right) }{\Gamma \left( q_{j}-\vartheta \right) \Gamma \left(
p_{j}\right) }<\frac{\prod_{j=1}^{n}\zeta ^{\ast }\left( p_{j},x\right) }{%
\prod_{j=1}^{n}\zeta ^{\ast }\left( q_{j},x\right) }<\prod_{j=1}^{n}\frac{%
\Gamma \left( p_{j}-\theta \right) \Gamma \left( q_{j}\right) }{\Gamma
\left( q_{j}-\theta \right) \Gamma \left( p_{j}\right) }
\label{PaHzpj/aHzqj<>}
\end{equation}%
holds for all $x>0$, where the lower bounds in (\ref{PaHzpj/aHzqj<>}) for $%
\vartheta =0$ is the best possible.
\end{proposition}

Taking $\left( \vartheta ,\theta \right) =\left( 0,\theta _{0}\right) $, $%
\left( p_{1},p_{2}\right) =\left( \left( p+q\right) /2,\left( p+q\right)
/2\right) $ and $\left( q_{1},q_{2}\right) =\left( p,q\right) $ in \ref%
{PaHzpj/aHzqj<>} gives the following Tur\'{a}n type inequalities:%
\begin{equation}
1<\frac{\zeta ^{\ast }\left( \left( p+q\right) /2,x\right) ^{2}}{\zeta
^{\ast }\left( p,x\right) \zeta ^{\ast }\left( q,x\right) }<\frac{\Gamma
\left( \left( p+q\right) /2-\theta _{0}\right) ^{2}\Gamma \left( p\right)
\Gamma \left( q\right) }{\Gamma \left( \left( p+q\right) /2\right)
^{2}\Gamma \left( p-\theta _{0}\right) \Gamma \left( q-\theta _{0}\right) }
\label{PaHzpj/aHzqj<>s}
\end{equation}%
for $x>0$. This implies the following log-convexity property.

\begin{corollary}
\label{C-P2}For fixed $x>0$, the functions $p\mapsto \zeta ^{\ast }\left(
p,x\right) $ and%
\begin{equation*}
p\mapsto \frac{\Gamma \left( p\right) }{\Gamma \left( p-\theta _{0}\right) }%
\zeta ^{\ast }\left( p,x\right)
\end{equation*}
are log-concave and log-convex on $\left( 1,\infty \right) $, respectively.
\end{corollary}

The alternating Hurwitz zeta function $\zeta ^{\ast }\left( p,x\right) $
contains two important special cases, that are,%
\begin{eqnarray*}
\zeta ^{\ast }\left( p,1\right)  &=&\sum_{k=1}^{\infty }\frac{\left(
-1\right) ^{k-1}}{k^{p}}=\eta \left( p\right) , \\
2^{p}\zeta ^{\ast }\left( p,\frac{1}{2}\right)  &=&\sum_{k=0}^{\infty }\frac{%
\left( -1\right) ^{k}}{\left( 2k+1\right) ^{p}}=\beta \left( p\right) ,
\end{eqnarray*}%
where $\eta \left( p\right) $ and $\beta \left( p\right) $ are the Dirichlet
eta function and Dirichlet beta function \cite{Abramowitz-HMFFGMT-1972},
respectively. There exist the relations between they and Hurwitz zeta
function:%
\begin{eqnarray*}
\eta \left( p\right)  &=&\left( 1-2^{1-p}\right) \zeta \left( p,1\right)
=\left( 1-2^{1-p}\right) \zeta \left( p\right) , \\
\beta \left( p\right)  &=&\frac{1}{4^{p}}\left( \zeta \left( p,\frac{1}{4}%
\right) -\zeta \left( p,\frac{3}{4}\right) \right) ,
\end{eqnarray*}%
where $\zeta \left( p\right) =\zeta \left( p,1\right) $ is the Riemann zeta
function. In 1998, Wang \cite{Wang-JCCU-14-1998} proved that $p\mapsto \eta
\left( p\right) $ is log-concave on $\left( 0,\infty \right) $. The
concavity property of $\eta \left( p\right) $ on $\left( 0,\infty \right) $
was due to Alzer and Kwong \cite[Theorem 3.1]{Alzer-JNT-151-2015}. A more
general result that $p\mapsto x^{p}\zeta ^{\ast }\left( p,x\right) $ is
concave on $\left( 0,\infty \right) $ was proved by \cite[Theorem (i)]%
{Cvijovic-JMAA-487-2020}. In particular, $p\mapsto a^{-p}\zeta ^{\ast
}\left( p,1/a\right) =\eta _{a}\left( p\right) $ is concave on $\left(
0,\infty \right) $, which was proved in \cite[Theorem 1.1]%
{Adell-JNT-157-2015}.

\begin{remark}
\label{R1-CP2}Now, taking $x=1$ in Corollary \ref{C-P2} yields the
conclusion that $p\mapsto \zeta ^{\ast }\left( p,1\right) =\eta \left(
p\right) $ is log-concave on $\left( 0,\infty \right) $. This gives a new
and simple proof of \cite{Wang-JCCU-14-1998}.
\end{remark}

\begin{remark}
\label{R2-CP2}In a similar way to Remark \ref{R-CP1}, the log-concavity of $%
p\mapsto \zeta ^{\ast }\left( p,x\right) $ implies that so is $p\mapsto
x^{pa+b}\zeta ^{\ast }\left( p,x\right) $ for any $a,b\in \mathbb{R}$.
Hence, $p\mapsto x^{p}\zeta ^{\ast }\left( p,x\right) $ is also log-concave
on $\left( 0,\infty \right) $. This conclusion is weaker than \cite[Theorem
(i)]{Cvijovic-JMAA-487-2020}. In the same way, the function $p\mapsto \Gamma
\left( p\right) x^{p}\zeta ^{\ast }\left( p,x\right) /\Gamma \left( p-\theta
_{0}\right) $ is log-convex on $\left( 0,\infty \right) $.
\end{remark}

It was listed in \cite[Eq. (23.2.22)]{Abramowitz-HMFFGMT-1972} that
\begin{subequations}
\begin{equation}
\beta \left( 2n+1\right) =\frac{\left( \pi /2\right) ^{2n+1}}{2\left(
2n\right) !}\left\vert E_{2n}\right\vert \text{ \ for }n\in \mathbb{N}_{0},
\label{b2n+1=}
\end{equation}%
where $E_{n}$ is the Euler numbers defined by
\end{subequations}
\begin{equation*}
\frac{1}{\cosh t}=\sum_{k=0}^{\infty }E_{n}\frac{t^{n}}{n!}\text{, \ }%
\left\vert t\right\vert <\frac{\pi }{2}.
\end{equation*}%
By the log-concavity of $x^{p}\zeta ^{\ast }\left( p,x\right) $ and
log-convexity of $\Gamma \left( p\right) x^{p}\zeta ^{\ast }\left(
p,x\right) /\Gamma \left( p-\theta _{0}\right) $ in $p$, we can derive the
bounds for the ratio of two non-zero neighboring Euler numbers.

\begin{corollary}
The double inequalities%
\begin{eqnarray*}
\frac{8}{\pi ^{2}}n\left( 2n-1\right)  &<&\frac{\left\vert E_{2n}\right\vert
}{\left\vert E_{2n-2}\right\vert }<n\left( 2n-1\right)  \\
\frac{1}{15}\left( 4n+1\right) \left( 4n-1\right)  &<&\frac{\left\vert
E_{2n}\right\vert }{\left\vert E_{2n-2}\right\vert }<\frac{1}{\pi ^{2}}%
\left( 4n+1\right) \left( 4n-1\right)
\end{eqnarray*}%
hold for $n\in \mathbb{N}$.
\end{corollary}

\begin{proof}
By Remark \ref{R2-CP2}, we see that the functions $\beta \left( p\right) $
and $\Gamma \left( p\right) \beta \left( p\right) /\Gamma \left( p-\theta
_{0}\right) $ are log-concave and log-convex on $\left( 0,\infty \right) $,
respectively. Then the sequences%
\begin{equation*}
\frac{\beta \left( 2n+1\right) }{\beta \left( 2n-1\right) }\text{ \ and \ }%
\frac{\Gamma \left( 2n+1\right) /\Gamma \left( 2n+1-\theta _{0}\right) }{%
\Gamma \left( 2n-1\right) /\Gamma \left( 2n-1-\theta _{0}\right) }\frac{%
\beta \left( 2n+1\right) }{\beta \left( 2n-1\right) }
\end{equation*}%
are decreasing and increasing for $n\in \mathbb{N}$, which imply that%
\begin{eqnarray*}
1 &<&\frac{\beta \left( 2n+1\right) }{\beta \left( 2n-1\right) }<\frac{\beta
\left( 3\right) }{\beta \left( 1\right) }=\frac{\pi ^{3}/32}{\pi /4}=\frac{%
\pi ^{2}}{8} \\
\frac{2}{\left( 1-\theta _{0}\right) \left( 2-\theta _{0}\right) }\frac{\pi
^{2}}{8} &<&\frac{2n\left( 2n-1\right) }{\left( 2n-\theta _{0}-1\right)
\left( 2n-\theta _{0}\right) }\frac{\beta \left( 2n+1\right) }{\beta \left(
2n-1\right) }<1.
\end{eqnarray*}%
Substituting (\ref{b2n+1=}) into the above inequalities then simplifying
gives%
\begin{eqnarray*}
\frac{8}{\pi ^{2}}n\left( 2n-1\right)  &<&\frac{\left\vert E_{2n}\right\vert
}{\left\vert E_{2n-2}\right\vert }<n\left( 2n-1\right) , \\
\frac{\left( 2n-\theta _{0}-1\right) \left( 2n-\theta _{0}\right) }{\left(
\theta _{0}-1\right) \left( \theta _{0}-2\right) } &<&\frac{\left\vert
E_{2n}\right\vert }{\left\vert E_{2n-2}\right\vert }<\frac{4}{\pi ^{2}}%
\left( 2n-\theta _{0}-1\right) \left( 2n-\theta _{0}\right)
\end{eqnarray*}%
Since $\theta _{0}=-0.439\,228...>-1/2$, we have $\left( 2n-\theta
_{0}\right) \left( 2n-\theta _{0}-1\right) <\left( 2n+1/2\right) \left(
2n-1/2\right) $ and%
\begin{equation*}
\frac{\left( 2n-\theta _{0}-1\right) \left( 2n-\theta _{0}\right) }{\left(
\theta _{0}-1\right) \left( \theta _{0}-2\right) }>\frac{4}{15}\left( 2n-%
\frac{1}{2}\right) \left( 2n+\frac{1}{2}\right) .
\end{equation*}%
Then the required double inequalities follow, and the proof is complete.
\end{proof}

\subsection{Application to the Tricomi function: Extension of \protect\cite[%
Proposition 6.1]{Yang-JMAA-551-2025}}

Consider the function%
\begin{equation}
U_{p}\left( x\right) =U_{p}\left( a,b,x\right) =\frac{1}{\Gamma \left(
a\right) }\int_{0}^{\infty }t^{p}t^{a-1}\left( 1+t\right) ^{b-a-1}e^{-xt}dt,
\label{Up}
\end{equation}%
where $p>-a$, $b\in \mathbb{R}$. Clearly, $U_{0}\left( x\right) =U_{0}\left(
a,b,x\right) $ is the confluent hypergeometric function of the second kind,
also named as Tricomi function. By the differentiation formula \cite[Eq.
(13.4.22)]{Abramowitz-HMFFGMT-1972}%
\begin{equation*}
\left( -1\right) ^{n}\frac{d^{n}}{dx^{n}}U\left( a,b,x\right) =\left(
a\right) _{n}U\left( a+n,b+n,x\right) ,
\end{equation*}%
we see that%
\begin{equation}
U_{n}\left( x\right) :=\left( -1\right) ^{n}U^{\left( n\right) }\left(
x\right) =\left( a\right) _{n}U\left( a+n,b+n,x\right)  \label{Un}
\end{equation}%
for $n\in \mathbb{N}_{0}$. For any $p>-a$, by (\ref{Up}), we have%
\begin{equation*}
U_{p}\left( x\right) =\frac{\Gamma \left( a+p\right) }{\Gamma \left(
a\right) }U\left( a+p,b+p,x\right) .
\end{equation*}%
Thus, by Theorem \ref{T1}, \cite[Proposition 6.1]{Yang-JMAA-551-2025} is
still true when replacing $\boldsymbol{p}_{\left[ k\right] },\boldsymbol{q}_{%
\left[ k\right] }\in \mathbb{N}_{0}^{k}$ by $\boldsymbol{p}_{\left[ k\right]
},\boldsymbol{q}_{\left[ k\right] }\in \left( -a,\infty \right) ^{k}$. For
the sake of completeness, we record it as follows.

\begin{proposition}
\label{P-U-cm} Let%
\begin{equation}
\mathcal{U}_{\rho }\left( x\right) =\prod_{j=1}^{n}U\left(
a+p_{j},b+p_{j},x\right) -\rho \prod_{j=1}^{n}U\left(
a+q_{j},b+q_{j},x\right) \text{.}  \label{Ur}
\end{equation}%
Suppose that the $n$-tuples $\boldsymbol{p}_{\left[ n\right] }=\left(
p_{1},...,p_{n}\right) $ and $\boldsymbol{q}_{\left[ n\right] }=\left(
q_{1},...,q_{n}\right) \in \left( -a,\infty \right) ^{n}$ for $n\geq 2$
satisfy $\boldsymbol{p}_{\left[ n\right] }\prec \boldsymbol{q}_{\left[ n%
\right] }$. Then the follows statements are valid.

(i) If $a>0$ and $b>1$ with $a-b+1>\left( <\right) 0$, then the function $%
x\mapsto +\left( -\right) \mathcal{U}_{\rho }\left( x\right) $ is CM on $%
\left( 0,\infty \right) $ if and only if%
\begin{equation*}
\rho \leq \left( \geq \right) \lambda _{n}^{\ast }=\prod\limits_{j=1}^{n}%
\frac{\Gamma \left( b+p_{j}-1\right) \Gamma \left( a+q_{j}\right) }{\Gamma
\left( b+q_{j}-1\right) \Gamma \left( a+p_{j}\right) }.
\end{equation*}
Consequently, the sharp inequality%
\begin{equation}
\lambda _{n}^{\ast }<\left( >\right) \prod_{j=1}^{n}\frac{U\left(
a+p_{j},b+p_{j},x\right) }{U\left( a+q_{j},b+q_{j},x\right) }
\label{PUpi/Uqi><lk*}
\end{equation}%
holds for $x>0$.

(ii) If $a>0$ and $a-b+1>\left( <\right) 0$, then the function $x\mapsto
-\left( +\right) \mathcal{U}_{\rho }\left( x\right) $ is CM on $\left(
0,\infty \right) $ if and only if $\rho \geq \left( \leq \right) 1$.
Therefore, the sharp inequality%
\begin{equation}
\prod_{j=1}^{n}\frac{U\left( a+p_{j},b+p_{j},x\right) }{U\left(
a+q_{j},b+q_{j},x\right) }<\left( >\right) 1  \label{PUpi/Uqi<>1}
\end{equation}%
holds for $x>0$.
\end{proposition}

\begin{remark}
Clearly, \cite[Corollaries 6.1--6.3]{Yang-JMAA-551-2025} are still valid.
\end{remark}

In a similar way to Corollaries \ref{C-P1} and \ref{C-P2}, we have the
following log-convexity property.

\begin{corollary}
Let $x>0$. The following statements are valid.

(i) If $a>0$ and $b>1$ with $a-b+1>\left( <\right) 0$, then the function%
\begin{equation*}
p\mapsto \frac{\Gamma \left( a+p\right) }{\Gamma \left( b+p-1\right) }%
U\left( a+p,b+p,x\right)
\end{equation*}
is log-concave (-convex) on $\left( -a,\infty \right) $.

(ii) If $a>0$ and $a-b+1>\left( <\right) 0$, then the function $p\mapsto
U\left( a+p,b+p,x\right) $ is log-convex (-concave) on $\left( -a,\infty
\right) $.
\end{corollary}

\subsection{Application to the Mills ratio}

Let $\phi \left( x\right) =e^{-x^{2}/2}/\sqrt{2\pi }$ be the density
function of a standard Gaussian law and $\Phi \left( x\right) =\int_{-\infty
}^{x}\phi \left( t\right) dt$ is its cumulative distribution function. The
function%
\begin{equation*}
R\left( x\right) =\frac{1-\Phi \left( x\right) }{\phi \left( x\right) }%
=e^{x^{2}/2}\int_{x}^{\infty }e^{-t^{2}/2}dt=\int_{0}^{\infty
}e^{-s^{2}/2}e^{-xs}ds,\text{, \ }x\geq 0
\end{equation*}%
is called Mills ratio \cite{Mills-B-18-1926}. Several relevant results for
the Mills ratio can refer to \cite{Kouba-arXiv-0607694v1}, \cite%
{Baricz-JMAA-340-2008}, \cite{Gasull-JMAA-420-2014}, \cite%
{From-JMAA-486-2020}, \cite[Theorem 1]{Yang-SP-66-2025}.

As the fourth application of Theorem \ref{T1}, we consider the function%
\begin{equation}
R_{p}\left( x\right) =\int_{0}^{\infty }t^{p}e^{-t^{2}/2}e^{-xt}dt\text{ \
for }x>0\text{ and }p>-1,  \label{Rp}
\end{equation}%
which satisfies%
\begin{equation*}
\left( -1\right) ^{p-1}R^{\left( p\right) }\left( x\right) =R_{p}\left(
x\right) \text{ \ for }p\in \mathbb{N}_{0}\text{.}
\end{equation*}%
Therefore, the function $R_{p}\left( x\right) $ can also be called
\textquotedblleft Mills ratio of order $p$\textquotedblright .

We claim that $R_{p}\left( x\right) $ converges on $\left( 0,\infty \right) $
for $p>-1$. In fact, from the inequality%
\begin{equation*}
R_{p}\left( x\right) =\int_{0}^{\infty
}t^{p}e^{-t^{2}/2}e^{-xt}dt<\int_{0}^{\infty }t^{p}e^{-xt}dt
\end{equation*}%
and the fact that $\int_{0}^{\infty }t^{p}e^{-xt}dt=\Gamma \left( p+1\right)
/x^{p+1}$ converges on $\left( 0,\infty \right) $ for $p>-1$, the
convergence of $R_{p}\left( x\right) $ on $\left( 0,\infty \right) $ for $%
p>-1$ follows.

Now, we give the asymptotic behavior of $R_{p}\left( x\right) $ when $x$
approaches zero or is sufficiently large, which is needed to prove
Proposition \ref{P-Rp-cm}.

\begin{lemma}
Let $p>-1$. (i) The function $R_{p}\left( x\right) $ defined by (\ref{Rp})
has the following power series representation:%
\begin{equation}
R_{p}\left( x\right) =\sum_{k=0}^{\infty }\frac{\left( -1\right)
^{k}2^{\left( k+p-1\right) /2}}{k!}\Gamma \left( \frac{k+p+1}{2}\right) x^{k}%
\text{ \ for }x>0\text{.}  \label{Rp-psr}
\end{equation}%
(ii) The function $R_{p}\left( x\right) $ has the asymptotic expansion:%
\begin{equation}
R_{p}\left( x\right) \sim \sum_{k=0}^{\infty }\frac{\left( -1\right)
^{k}\Gamma \left( 2k+p+1\right) }{2^{k}k!x^{2k+p+1}}\text{ \ as }%
x\rightarrow \infty \text{.}  \label{Rp-af-00}
\end{equation}
\end{lemma}

\begin{proof}
(i) Expanding in power series gives%
\begin{eqnarray*}
R_{p}\left( x\right)  &=&\int_{0}^{\infty }\left(
t^{p}e^{-t^{2}/2}\sum_{k=0}^{\infty }\left( \frac{\left( -1\right) ^{k}}{k!}%
t^{k}x^{k}\right) \right) dt=\sum_{k=0}^{\infty }\frac{\left( -1\right) ^{k}%
}{k!}\left( \int_{0}^{\infty }t^{k+p}e^{-t^{2}/2}dt\right) x^{k} \\
&=&\sum_{k=0}^{\infty }\frac{\left( -1\right) ^{k}}{k!}2^{\left(
k+p-1\right) /2}\Gamma \left( \frac{k+p+1}{2}\right) x^{k},
\end{eqnarray*}%
where the last equality follows from%
\begin{equation*}
\int_{0}^{\infty }t^{k+p}e^{-t^{2}/2}dt\overset{t=\sqrt{2s}}{=\!=\!=\!=\!=}%
\frac{1}{\sqrt{2}}\int_{0}^{\infty }2^{\left( k+p\right) /2}s^{\left(
k+p-1\right) /2}e^{-s}ds=2^{\left( k+p-1\right) /2}\Gamma \left( \frac{k+p+1%
}{2}\right) .
\end{equation*}%
It is easy to verify that the power series (\ref{Rp-psr}) converges on $%
\left( 0,\infty \right) $ for $p>-1$.

(ii) Since%
\begin{equation*}
t^{p}e^{-t^{2}/2}=\sum_{k=0}^{\infty }\frac{\left( -1\right) ^{k}}{2^{k}k!}%
t^{2k+p}=\sum_{k=0}^{\infty }\frac{\left( -1\right) ^{k}}{2^{k}k!}t^{\left(
k+\lambda -\mu \right) /\mu }\text{, \ as }t\rightarrow 0^{+},
\end{equation*}%
where $\lambda =\left( p+1\right) /2$ and $\mu =1/2$, by Watson's Lemma \cite%
[\S 2.3(ii)]{NIST-DLMF-Release 1.2.2}, it follows that%
\begin{equation*}
R_{p}\left( x\right) =\int_{0}^{\infty }t^{p}e^{-t^{2}/2}e^{-xt}dt\sim
\sum_{k=0}^{\infty }\Gamma \left( \frac{k+\lambda }{\mu }\right) \frac{%
\left( -1\right) ^{k}/\left( 2^{k}k!\right) }{x^{\left( k+\lambda \right)
/\mu }}=\sum_{k=0}^{\infty }\frac{\left( -1\right) ^{k}\Gamma \left(
2k+p+1\right) }{2^{k}k!x^{2k+p+1}}
\end{equation*}%
as $x\rightarrow \infty $, which completes the proof.
\end{proof}

\begin{remark}
If $p=n\in \mathbb{N}_{0}$, then the power series representation (\ref%
{Rp-psr}) reduces to%
\begin{equation*}
R_{n}\left( x\right) =\left( -1\right) ^{n}R^{\left( n\right) }\left(
x\right) =\sum_{k=0}^{\infty }\frac{\left( -1\right) ^{k}2^{\left(
k+n-1\right) /2}}{k!}\Gamma \left( \frac{k+n+1}{2}\right) x^{k},
\end{equation*}
and the asymptotic expansion (\ref{Rp-af-00}) reduces to \cite[Eq. (1.3)]%
{Yang-SP-66-2025}.
\end{remark}

Now let $f\left( t\right) =e^{-t^{2}/2}$. Obviously, the function $f_{\theta
}\left( t\right) =t^{\theta }f\left( t\right) =t^{\theta }e^{-t^{2}/2}$ is
log-concave on $\left( 0,\infty \right) $ if and only if $\theta \geq 0$,
and is log-convex on $\left( 0,\infty \right) $ if and only if $\theta
\rightarrow -\infty $. By Theorem \ref{T1} and Remark \ref{R-T1}, we have
the following proposition.

\begin{proposition}
\label{P-Rp-cm}Let $R_{p}\left( x\right) $ for $p>-1$ be defined by (\ref{Rp}%
). If the $n$-tuples $\boldsymbol{p}_{\left[ n\right] }=\left(
p_{1},...,p_{k}\right) $ and $\boldsymbol{q}_{\left[ n\right] }=\left(
q_{1},...,q_{k}\right) \in \left( -1,\infty \right) ^{n}$ for $n\geq 2$
satisfy $\boldsymbol{p}_{\left[ n\right] }\prec \boldsymbol{q}_{\left[ n%
\right] }$, then the function%
\begin{equation*}
x\mapsto \mathcal{D}_{\boldsymbol{p}_{\left[ n\right] },\boldsymbol{q}_{%
\left[ n\right] }}\left( x;\lambda _{n}\right)
=\prod_{j=1}^{n}R_{p_{j}}\left( x\right) -\lambda
_{n}\prod_{j=1}^{n}R_{q_{j}}\left( x\right)
\end{equation*}%
is CM on $\left( 0,\infty \right) $ if and only if%
\begin{equation*}
\lambda _{n}\leq \lambda _{n}^{\left[ 0\right] }=\prod_{j=1}^{n}\frac{\Gamma
\left( p_{j}+1\right) }{\Gamma \left( q_{j}+1\right) };
\end{equation*}%
while $x\mapsto -\mathcal{D}_{\boldsymbol{p}_{\left[ n\right] },\boldsymbol{q%
}_{\left[ n\right] }}\left( x;\lambda _{n}\right) $ is CM on $\left(
0,\infty \right) $ if $\lambda _{n}\geq 1$. Consequently, the inequality%
\begin{equation}
\prod_{j=1}^{n}\frac{\Gamma \left( p_{j}+1\right) }{\Gamma \left(
q_{j}+1\right) }<\prod_{j=1}^{n}\frac{R_{p_{j}}\left( x\right) }{%
R_{q_{j}}\left( x\right) }<1  \label{pRpj/Rqj<>}
\end{equation}%
holds for $x>0$, where lower bound is sharp.
\end{proposition}

\begin{proof}
(i) The necessity follows from $\mathcal{D}_{\boldsymbol{p}_{\left[ n\right]
},\boldsymbol{q}_{\left[ n\right] }}\left( x;\lambda _{n}\right) \geq 0$ for
all $x>0$, which implies that%
\begin{equation}
\lambda _{n}\leq \lim_{x\rightarrow \infty }\prod_{j=1}^{n}\frac{%
R_{p_{j}}\left( x\right) }{R_{q_{j}}\left( x\right) }=\prod_{j=1}^{n}\frac{%
\Gamma \left( p_{j}+1\right) }{\Gamma \left( q_{j}+1\right) }=\lambda _{n}^{%
\left[ 0\right] },  \label{ln<00}
\end{equation}%
where we have used asymptotic expansion (\ref{Rp-af-00}).

To prove the sufficiency, we write%
\begin{equation*}
\mathcal{D}_{\boldsymbol{p}_{\left[ n\right] },\boldsymbol{q}_{\left[ n%
\right] }}\left( x;\lambda _{n}\right) =\mathcal{D}_{\boldsymbol{p}_{\left[ n%
\right] },\boldsymbol{q}_{\left[ n\right] }}\left( x;\lambda _{n}^{\left[ 0%
\right] }\right) +\left( \lambda _{n}^{\left[ 0\right] }-\lambda _{n}\right)
\prod_{j=1}^{n}R_{q_{j}}\left( x\right) .
\end{equation*}%
Since $\mathcal{D}_{\boldsymbol{p}_{\left[ n\right] },\boldsymbol{q}_{\left[
n\right] }}\left( x;\lambda _{n}^{\left[ 0\right] }\right) $ is CM on $%
\left( 0,\infty \right) $ by Theorem \ref{T1} and $R_{q_{j}}\left( x\right) $
is CM on $\left( 0,\infty \right) $, so is $\mathcal{D}_{\boldsymbol{p}_{%
\left[ n\right] },\boldsymbol{q}_{\left[ n\right] }}\left( x;\lambda
_{n}\right) $ on $\left( 0,\infty \right) $.

(ii) In a similar way, we can prove $x\mapsto -\mathcal{D}_{\boldsymbol{p}_{%
\left[ n\right] },\boldsymbol{q}_{\left[ n\right] }}\left( x;\lambda
_{n}\right) $ is CM on $\left( 0,\infty \right) $ if $\lambda _{n}\geq 1$.

(iii) The required double inequality follows from the inequalities $\mathcal{%
D}_{\boldsymbol{p}_{\left[ n\right] },\boldsymbol{q}_{\left[ n\right]
}}\left( x;\lambda _{n}^{\left[ 0\right] }\right) >0$ and $-\mathcal{D}_{%
\boldsymbol{p}_{\left[ n\right] },\boldsymbol{q}_{\left[ n\right] }}\left(
x;1\right) >0$ for all $x>0$. Due to the limit relation in (\ref{ln<00}),
the lower bound is sharp. This completes the proof.
\end{proof}

\begin{remark}
Taking $\left( p_{1},p_{2}\right) =\left( n,n\right) $ and $\left(
q_{1},q_{2}\right) =\left( n+1,n-1\right) $ for $n\in \mathbb{N}$ in (\ref%
{pRpj/Rqj<>}) yields \cite[Theorem 3.2]{Baricz-JMAA-340-2008}.
\end{remark}

In a similar way to Corollaries \ref{C-P1} and \ref{C-P2}, we have the
following log-convexity property.

\begin{corollary}
For fixed $x>0$, the functions $p\mapsto R_{p}\left( x\right) $ and $%
p\mapsto R_{p}\left( x\right) /\Gamma \left( p+1\right) $ are log-convex and
log-concave on $\left( -1,\infty \right) $, respectively.
\end{corollary}

\begin{remark}
Recently, Yang and Tian \cite[Theorem 1]{Yang-SP-66-2025} proved that the
double inequality%
\begin{equation}
\frac{\Gamma \left( n/2\right) \Gamma \left( n/2+1\right) }{\Gamma \left(
n/2+1/2\right) ^{2}}<\frac{R_{n-1}\left( x\right) R_{n+1}\left( x\right) }{%
R_{n}\left( x\right) ^{2}}<\frac{n+1}{n}  \label{Y-I1}
\end{equation}%
holds for $x>0$ and $n\in \mathbb{N}$. This implies that the sequence $%
n\mapsto R_{n}\left( x\right) /\Gamma \left( n/2+1/2\right) $ is log-convex
for $n\in \mathbb{N}$. Inspired by this, we guess that the function%
\begin{equation*}
p\mapsto \frac{1}{\Gamma \left( p/2+1/2\right) }R_{p}\left( x\right)
\end{equation*}%
is log-convex on $\left( -1,\infty \right) $.
\end{remark}

\section{Conclusions}

In this paper, we extended \cite[Theorem 1.1]{Yang-JMAA-551-2025} to Theorem %
\ref{T1} and corrected the errors in the proof of \cite[Theorem 1.1]%
{Yang-JMAA-551-2025}. Precisely, Theorem \ref{T1} holds not only for the $n$%
-tuples $\boldsymbol{p}_{\left[ n\right] }=\left( p_{1},...,p_{n}\right) $
and $\boldsymbol{q}_{\left[ n\right] }=\left( q_{1},...,q_{n}\right) \in
\mathbb{N}_{0}^{n}$ but also for $\boldsymbol{p}_{\left[ n\right] },%
\boldsymbol{q}_{\left[ n\right] }\in \mathbb{I}^{n}$, where $\mathbb{I}%
\subseteq \mathbb{R}$ is an interval. Correspondingly, \cite[Propositions
4.1, 5.1 and 6.1]{Yang-JMAA-551-2025} are also valid for $\boldsymbol{p}_{%
\left[ n\right] },\boldsymbol{q}_{\left[ n\right] }\in \left( 1,\infty
\right) ^{n}$, $\left( 0,\infty \right) ^{n}$ and $\left( -a,\infty \right)
^{n}$, respectively. In other words, \cite[Propositions 4.1, 5.1 and 6.1]%
{Yang-JMAA-551-2025} have be extended to Propositions \ref{P-gTd-cm-Hz}, \ref%
{P-gTd-cm-aHz} and \ref{P-U-cm}, respectively, where Propositions \ref%
{P-gTd-cm-Hz} and \ref{P-gTd-cm-aHz} provide applications of Theorem \ref{T1}
in the Hurwitz zeta and alternating Hurwitz zeta functions. Moreover,
Proposition \ref{P-Rp-cm} offers an application in probability theory.

Finally, we emphasize that Lemmas \ref{L-p,qn*,p,q2'} and \ref{Crit} play
important role in the proof of Theorem \ref{T1}, which can also be used to
deal with a class of problems involving the majorization of vectors. As an
example, we give a new proof of the following Hardy-Littlewood-Polya
inequality by Lemma \ref{L-p,qn*,p,q2'}.

\begin{theorem}[{\protect\cite[Theorem 108]{Hardy-I-CUP-1952}}]
Let $\phi $ be a continuous convex on the interval $I$. If the $n$-tuples $%
\boldsymbol{p}_{\left[ n\right] }=\left( p_{1},...,p_{n}\right) $ and $%
\boldsymbol{q}_{\left[ n\right] }=\left( q_{1},...,q_{n}\right) \in \mathbb{I%
}^{n}$ satisfy $\boldsymbol{p}_{\left[ n\right] }\prec \boldsymbol{q}_{\left[
n\right] }$, then%
\begin{equation*}
\sum_{j=1}^{n}\phi \left( p_{j}\right) \leq \sum_{j=1}^{n}\phi \left(
q_{j}\right) .
\end{equation*}
\end{theorem}

\begin{proof}
It suffices to prove%
\begin{equation*}
\Delta _{\boldsymbol{p}_{\left[ n\right] },\boldsymbol{q}_{\left[ n\right]
}}=\sum_{j=1}^{n}\phi \left( p_{j}\right) -\sum_{j=1}^{n}\phi \left(
q_{j}\right) \leq 0
\end{equation*}%
by induction. When $n=2$, $\boldsymbol{p}_{\left[ 2\right] }\prec
\boldsymbol{q}_{\left[ 2\right] }$ means that $q_{1}\geq p_{1}\geq p_{2}\geq
q_{2}$ with $p_{1}+p_{2}=q_{1}+q_{2}$. By the property of convex functions,%
\begin{equation*}
\frac{\phi \left( p_{2}\right) -\phi \left( q_{2}\right) }{p_{2}-q_{2}}\leq
\frac{\phi \left( q_{1}\right) -\phi \left( p_{1}\right) }{q_{1}-p_{1}},
\end{equation*}
which indicates that $\Delta _{\boldsymbol{p}_{\left[ 2\right] },\boldsymbol{%
q}_{\left[ 2\right] }}\leq 0$.

Assume that $\Delta _{\boldsymbol{p}_{\left[ n\right] },\boldsymbol{q}_{%
\left[ n\right] }}\leq 0$ for certain $n\geq 2$. To prove $\Delta _{%
\boldsymbol{p}_{\left[ n+1\right] },\boldsymbol{q}_{\left[ n+1\right] }}\leq
0$, we write%
\begin{eqnarray*}
\Delta _{\boldsymbol{p}_{\left[ n+1\right] },\boldsymbol{q}_{\left[ n+1%
\right] }} &=&\phi \left( p_{1}\right) -\phi \left( q_{k}\right) -\phi
\left( q_{k+1}\right) +\phi \left( q_{k}+q_{k+1}-p_{1}\right) \\
&&+\sum_{j=2}^{n}\phi \left( p_{j}^{\ast }\right) -\sum_{j=1}^{n}\phi \left(
q_{j}^{\ast }\right) =\Delta _{\boldsymbol{p}_{\left[ 2\right] }^{\prime },%
\boldsymbol{q}_{\left[ 2\right] }^{\prime }}+\Delta _{\boldsymbol{p}_{\left[
n\right] }^{\ast },\boldsymbol{q}_{\left[ n\right] }^{\ast }},
\end{eqnarray*}%
where $\left( p_{j},q_{j}\right) $ and $\left( p_{j}^{\prime },q_{j}^{\prime
}\right) $ are defined by (\ref{pj*,qj*}) and (\ref{p2',q2'}), respectively.
Using Lemma \ref{L-p,qn*,p,q2'} and taking the same steps as the proof of
Theorem \ref{T1}, we can prove $\Delta _{\boldsymbol{p}_{\left[ n+1\right] },%
\boldsymbol{q}_{\left[ n+1\right] }}\leq 0$. By induction, $\Delta _{%
\boldsymbol{p}_{\left[ n\right] },\boldsymbol{q}_{\left[ n\right] }}\leq 0$
for all $n\geq 2$, which completes the proof.
\end{proof}

\section{Declarations}

\textbf{Funding}: No funding was received to assist with the preparation of
this manuscript.

\textbf{Conflicts of interest/Competing interests}: The author has no
relevant financial or non-financial interests to disclose.

\textbf{Use of AI tools} Declaration: The author has not used Artificial
Intelligence (AI) tools in the creation of this paper.

\end{document}